\documentclass[11pt]{amsart}

 \usepackage{amsmath,amsthm,amsfonts,amssymb,verbatim}

 \usepackage{graphicx}
 \usepackage[all]{xy}
  \usepackage{pinlabel}

\def\co{\colon\thinspace}

\newcommand{\Br}{\mathbf{\Sigma}}

\newcommand{\sC}{\mathcal{C}}

\newcommand{\bZ}{\mathbb{Z}}
\newcommand{\bQ}{\mathbb{Q}}

\newcommand{\bF}{\mathbb{F}}

\newcommand{\pq}{\textstyle\frac{p}{q}}

\newcommand{\overzero}{\textstyle\frac{1}{0}}

\newcommand{\si}{\sigma}

\newcommand{\close}{\overline}

\newcommand{\half}{\frac{1}{2}}

\newcommand{\Khred}{\widetilde{\operatorname{Kh}}{}}
\newcommand{\CKhred}{\widetilde{\operatorname{CKh}}}
\newcommand{\rk}{\operatorname{rk}}

\newcommand{\zero}
	{\raisebox{-2pt}
	{\includegraphics[scale=0.085]{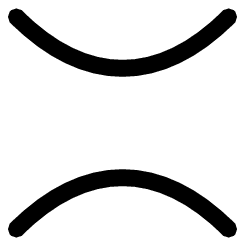}}}
\newcommand{\one}
	{\raisebox{-2pt}
	{\includegraphics[scale=0.085]{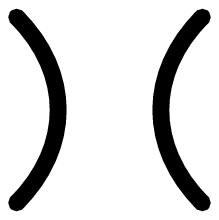}}}

\newcommand{\rightcross}
	{\raisebox{-2pt}
	{\includegraphics[scale=0.085]{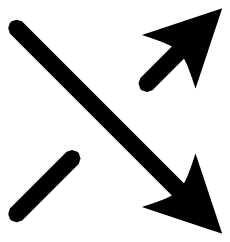}}}

\newtheorem*{namedtheorem}{\theoremname}
\newcommand{\theoremname}{testing}

\address[lwatson@math.ucla.edu]{Liam Watson, Department of Mathematics, University of California at Los Angeles, 520 Portola Plaza, Los Angeles 90095, USA}

\newtheorem{theorem}{Theorem}[section]
\newtheorem{corollary}[theorem]{Corollary}

\newtheorem{proposition}[theorem]{Proposition}

\newtheorem{claim}[theorem]{Claim}

\theoremstyle{definition}
\newtheorem{definition}[theorem]{Definition}

\newtheorem{remark}[theorem]{Remark}

\title[New proofs of certain finite filling results]{New proofs of certain finite filling results via Khovanov homology}

\author[Liam Watson]{Liam Watson}\thanks{Partially supported by an NSERC postdoctoral fellowship.}

\begin{document}

\maketitle

\begin{abstract}
We give a Khovanov homology proof that hyperbolic twist knots do not admit non-trivial Dehn surgeries with finite fundamental group. 
\end{abstract}



\section{Introduction}

Twists knots provide the simplest infinite family of hyperbolic knots. These arise by considering various twisted Whitehead doubles of the trivial knot, for example. Alternatively, and better suited to the purpose of this paper, let $K_t$ be the $(1,t+1,1)$-pretzel knot and consider the family $\{K_t\}$ where $t\ge 0$ (see Figure \ref{fig:twist-knots}). With this convention, $K_0$ is the left-handed trefoil knot and $K_1$ is the figure eight knot. For $t>0$ the knot $K_t$ is hyperbolic. 

\begin{figure}[ht!]
\begin{center}
\raisebox{0pt}{\includegraphics[scale=0.25]{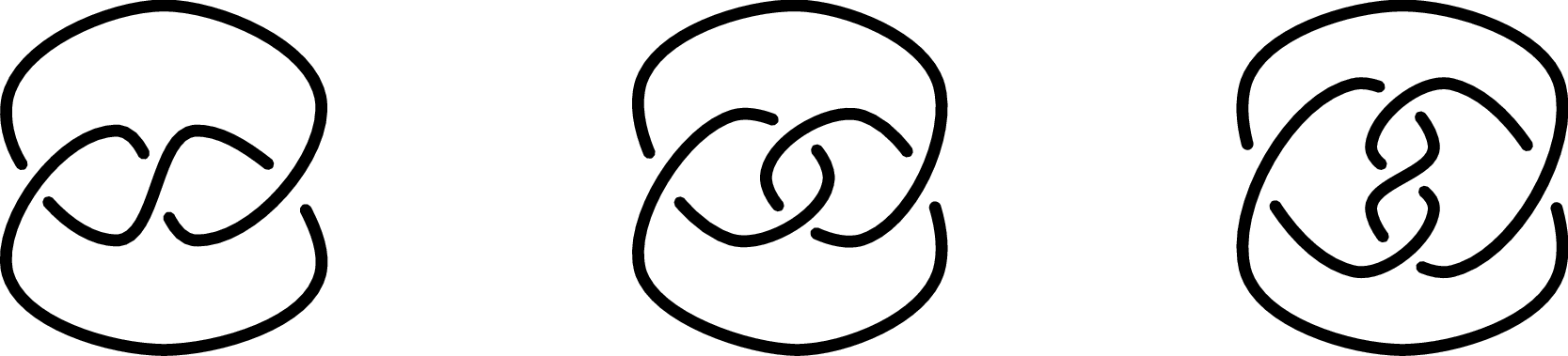}}\qquad\qquad
\labellist 
	\pinlabel $\vdots$ at 488 440
\endlabellist
\raisebox{-18pt}{\includegraphics[scale=0.25]{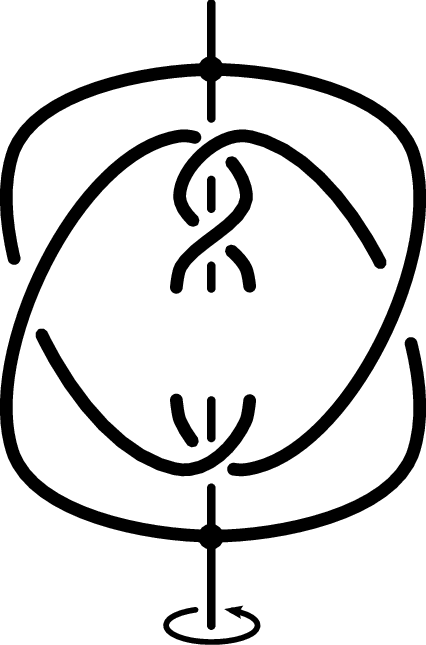}}\end{center}
\caption{The twist knots $K_0$, $K_1$ and $K_2$ (left), and a the general twist knot $K_t$ (right) with strong inversion indicated, where $t+1$ gives the number of vertical half twists.}
\label{fig:twist-knots}
\end{figure}

Given a knot $K$ in $S^3$, consider the three-manifold with torus boundary $M=S^3\smallsetminus\nu(K)$ where $\nu(K)$ is an open tubular neighbourhood of the knot. Define $S^3_r(K)=M\cup_h(D^2\times S^1)$ where $h\co \partial (D^2\times S^1) \to \partial M$ is the homeomorphism determined by $h(\partial D^2 \times \{{\rm point}\}) = p\mu+q\lambda$. In this notation, $\mu$ is the knot meridian, $\lambda$ is the Seifert longitude, and $r=\frac{p}{q}$ is an extended reduced rational number. The resulting closed three-manifold is the result of $r$-surgery on $K$; notice that the trivial surgery corresponds to the extended rational $\frac{1}{0}$. For background consistent with these conventions we refer the reader to Boyer \cite{Boyer2002}. 

A non-trivial surgery with finite fundamental group is called a finite filling. These fillings give a special type of exceptional surgery on a hyperbolic knot; a result of Thurston states that hyperbolic knots admit finitely many exceptional surgeries \cite{Thurston1980, Thurston1982}. The aim of this paper is to use Khovanov homology to prove the following:

\begin{theorem}[Delman \cite{Delman1995}, Tanguay \cite{Tanguay1996}]\label{thm:twist} Hyperbolic twist knots do not admit finite fillings.  \end{theorem}

This result was first proved by Delman using essential laminations \cite{Delman1995}, and independently by Tanguay via character variety methods \cite{Tanguay1996}. A complete classification of exceptional surgeries on two-bridge knots was subsequently obtained by Brittenham and Wu \cite{BW2001}. Note also that Boyer, Mattman and Zhang give a complete description of the fundamental polygon of any twist knot, from which Tanguay's proof may be recovered \cite{BMZ1997}. 

An alternate proof of Theorem \ref{thm:twist} may be obtained via Heegaard Floer homology. Since twist knots are alternating, it follows that hyperbolic twist knots do not admit L-space surgeries \cite[Theorem 1.5]{OSz2005-lens}. As manifolds with elliptic geometry are L-spaces \cite[Proposition 2.3]{OSz2005-lens}, Theorem \ref{thm:twist} follows. This appeals to an equivalence between manifolds with finite fundamental group and manifolds admitting elliptic geometry, though in the present setting geometrization in full generality is not required; the presence of a strong inversion (as described below) ensures that geometrization for orbifolds of cyclic type is sufficient \cite{BP2001,Thurston1982}.  

A strong inversion is an involution of a three-manifold with one-dimensional fixed point set. For example, every two-fold branched cover of $S^3$ admits a strong inversion. A knot $K$ in $S^3$ is strongly invertible if the standard strong inversion on $S^3$ induces a strong inversion on the knot complement $M=S^3\smallsetminus\nu(K)$. That is, $M$ admits an involution with one-dimensional fixed point set meeting $\partial M$ transversally in exactly 4 points. The twist knot $K_t$  is strongly invertible for all $t$; the relevant symmetry is exhibited in Figure \ref{fig:twist-knots}.

Given a strongly invertible knot, the involution on the knot complement may be extended to a strong inversion on any surgery \cite{Montesinos1975}. This is referred to as the Montesinos trick, and as a result every surgery on a strongly invertible knot may be realized as a two-fold branched cover of $S^3$ branched over a link. In this setting, the work in \cite{Watson2008} establishes that Khovanov homology may be used to provide obstructions  to finite fillings. This makes use of the Khovanov homology of the branch sets associated with a surgery on the knot via Montesinos' work. For example, Khovanov homology easily recovers the fact that the figure eight knot does not admit finite fillings \cite[Theorem 7.2]{Watson2008} (this result is originally due to Thurston \cite{Thurston1980}). The key observation is that, to a certain degree, relatively simple two-fold branched covers (in terms of geometry) have branch sets with simple Khovanov homology. 

Using the strong inversion on $K_t$, the goal of this paper is to apply the results in \cite{Watson2008} to prove Theorem \ref{thm:twist}. The proof is entirely combinatorial and new, in the sense that it does not appeal to any of the machinery described above (Heegaard Floer homology, character varieties, essential laminations, etc.). However, the relationship between Heegaard Floer homology and Khovanov homology suggests that the particular obstructions from these two theories may be related. Part of the motivation for pursuing a proof via Khovanov homology stems from an interest in comparing the obstructions from Heegaard Floer homology and from Khovanov homology. Further motivation is in establishing methods for applying obstructions from Khovanov homology to infinite families. This poses the immediate challenge of pushing calculation techniques beyond the limits imposed by machine calculation, and in turn is a central focus of this paper.

\subsection{The proof of Theorem \ref{thm:twist}} The feature of Khovanov homology exploited in this proof is homological width. For a given link $L$ this is the integer $w(L)$ recording the number of diagonals supporting non-trivial homology (see Definition \ref{def:width}).  Let $\Br(S^3,L)$ denote the two-fold branched cover of $S^3$ branched over a link $L$.  The following is proved in \cite{Watson2008}:

\begin{theorem}[{\cite[Theorem 6.3]{Watson2008}}]\label{thm:width} If $\Br(S^3,L)$ has finite fundamental group then $w(L)\le 2$. \end{theorem}

With this result in hand the aim of this paper is to establish:

\begin{theorem}\label{thm:combo} The surgery $S^3_r(K_t)$ may be realized as the two-fold branched cover $\Br(S^3,\tau_t(r))$, where the branch set $\tau_t(r)$ satisfies $w(\tau_t(r))\ge t+1$. \end{theorem}

When $t>1$, Theorem \ref{thm:twist} follows immediately from Theorem \ref{thm:width} and Theorem \ref{thm:combo}. The case $t=1$ corresponds to surgery on the figure eight knot and we appeal to the amphichirality of $K_1$ combined with the fact that the width bound of Theorem \ref{thm:width} for finite fillings is violated for non-negative surgery coefficients (compare \cite[Theorem 7.2]{Watson2008}).

The remainder of the paper is devoted to establishing Theorem \ref{thm:combo}, and is organized as follows. The proof has essentially two parts: first describe the exterior of each $K_t$ as the two-fold branched cover of a tangle (see Proposition \ref{prp:framed-tangle}), and second determine the coarse behaviour of the Khovanov homology (precisely, the homological width) of various rational closures of these tangles (see Proposition \ref{prp:branch-width}). The former is established in Section \ref{sec:branch} and is a straightforward application of the Montesinos trick, while calculations pertaining to the latter occupy Section \ref{sec:width}. 

A key observation is that relatively few integer surgeries need to be considered in order to infer the homological width for the infinite family of branch sets associated with each knot $K_t$. This exploits some particularly stable behaviour of the families of links that arise, as developed in \cite{Watson2008}. The sense in which the homological width might be considered coarse information is elucidated through these calculations. Indeed, we will only calculate the Khovanov homology for certain branch sets up to indeterminate summands (see Remark \ref{rmk:tetris}), thus the homological width is obtained without a full description of the Khovanov homology. Since Khovanov homology can be difficult to compute for knots with a high number of crossings, this suggests that it may be possible to calculate (or bound) the homological width in certain settings without needing to calculate the entire invariant. 

For the reader's reference, we have collected the requisite properties of reduced Khovanov homology with coefficients in $\bZ/2\bZ$ in Section \ref{sec:Kh}. This section constitutes the bulk of the paper, and may be skipped at first reading by those already familiar with the invariant.

\subsection*{Acknowledgements} Much of the work associated with this paper was completed while in residence at MSRI in 2010. The author thanks the organizers of the program {\em Homology Theories of Knots and Links} for providing a stimulating work environment. Thanks also to Tye Lidman for helpful discussions pertaining to spectral sequences, Paul Turner for comments on an earlier draft of this paper, and Jeremy Van Horn-Morris for valuable comments regarding open book decompositions. Finally, the detailed comments provided by the referees improved the exposition of the paper throughout. 

\section{Branch sets} \label{sec:branch}

Given a manifold obtained by Dehn surgery on a strongly invertible knot, the Montesinos trick provides an alternate description of the manifold as a two-fold branched cover of $S^3$ \cite{Montesinos1975}. To make this precise we review the notation introduced in \cite[Section 3]{Watson2008}.  When $K$ is strongly invertible, the knot exterior $S^3\smallsetminus\nu(K)$ is the two-fold branched cover of a tangle $T=(B^3,\tau)$ where $B^3$ is a three-ball and $\tau$ is a pair of properly embedded arcs. Denote this two-fold branched cover by  $\Br(B^3,\tau)$. The arcs $\tau$ are the image of the fixed point set in the quotient of the involution on the exterior $S^3\smallsetminus\nu(K)$. As a result, the tangle $T$ is well-defined up to homeomorphism of the pair $(B^3,\tau)$, and any diagram of $T$ may be regarded as a choice of representative for the homeomorphism class. 

The boundary of the knot exterior has a preferred generating set for homology given by the knot meridian $\mu$ and the Seifert longitude $\lambda$. It is always possible to make a choice of corresponding preferred representative for the homeomorphism class of the pair $(B^3,\tau)$ as illustrated in Figure \ref{fig:closures} (see \cite[Corollary 3.8]{Watson2008}, for example). That is, $S^3\cong \Br(S^3,\tau(\overzero))$ and $S^3_0(K)\cong \Br(S^3,\tau(0))$ provide explicit descriptions of the branch sets corresponding to the trivial and zero surgeries, respectively. This may be thought of as fixing a framing on the tangle $(B^3,\tau)$, just as $\lambda$ corresponds to the Seifert framing of the knot. Throughout this paper we will use $(B_3,\tau)$ to denote the preferred representative of the associated quotient tangle of a given strongly invertible knot. 

\begin{figure}[ht!]
\begin{center}
\labellist \small
	\pinlabel $\mu$ at -80 278
	\pinlabel $\lambda$ at -150 210
	\pinlabel $\tau(\overzero)$ at 280 210
	\pinlabel $\tau(0)$ at 620 210
\endlabellist
\raisebox{0pt}{\includegraphics[scale=0.25]{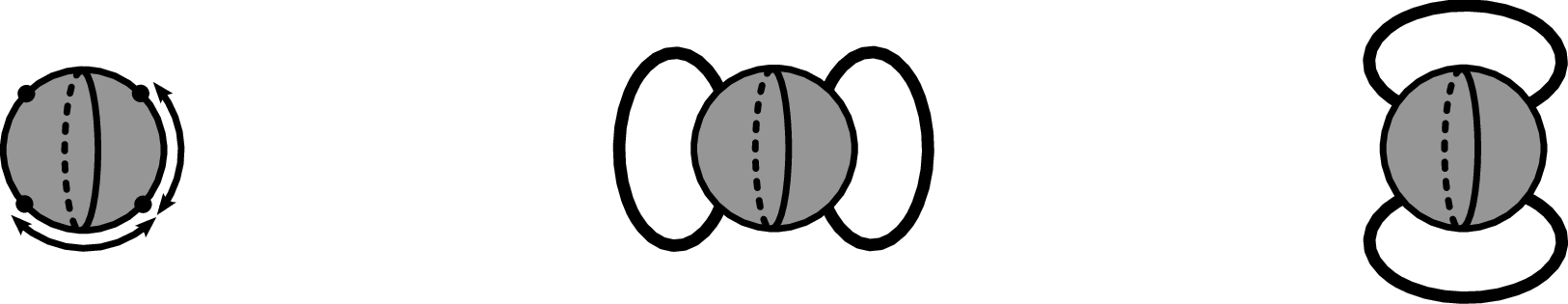}}
\end{center}
\caption{The arcs $\mu$ and $\lambda$ in the boundary of a tangle (left), labelled -- abusing notation -- by their respective lifts in the two-fold branched cover; and the closures $\tau(\overzero)$ and $\tau(0)$ corresponding to the trivial and zero surgeries in the cover, respectively.}
\label{fig:closures}
\end{figure}
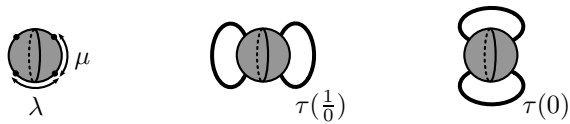

More generally, branch sets for the integer surgeries are given by the links $\tau(n)$ as in Figure \ref{fig:fraction} so that $S^3_n(K)\cong\Br(S^3,\tau(n))$. In particular, notice that the half twist in the branch set lifts to a full Dehn twist in the cover, so that $\tau(n)$ is the branch set for the manifold obtained by   filling along the slope $n\mu+\lambda$ in the boundary of the knot exterior. It follows that $\det(\tau(n))=|n|$ (recall that $\det(L)=|H_1(\Br(S^3,L);\bZ)|$ whenever $H_1(\Br(S^3,L);\bZ)$ is finite, and zero otherwise). It is also possible to describe $\pq$-surgery on a strongly invertible knot (having fixed a continued fraction expansion of the surgery coefficient) so that $S^3_{p/q}(K)\cong\Br(S^3,\tau(\pq))$; see \cite[Section 3]{Watson2008} for details.

\begin{figure}[ht!]
\begin{center}
\labellist \small 
	\pinlabel $\cdots$ at 382 464
	\pinlabel $\underbrace{\phantom{aaaaaaaaaaa}}_n$ at 325 395
\endlabellist	
\raisebox{0pt}{\includegraphics[scale=0.25]{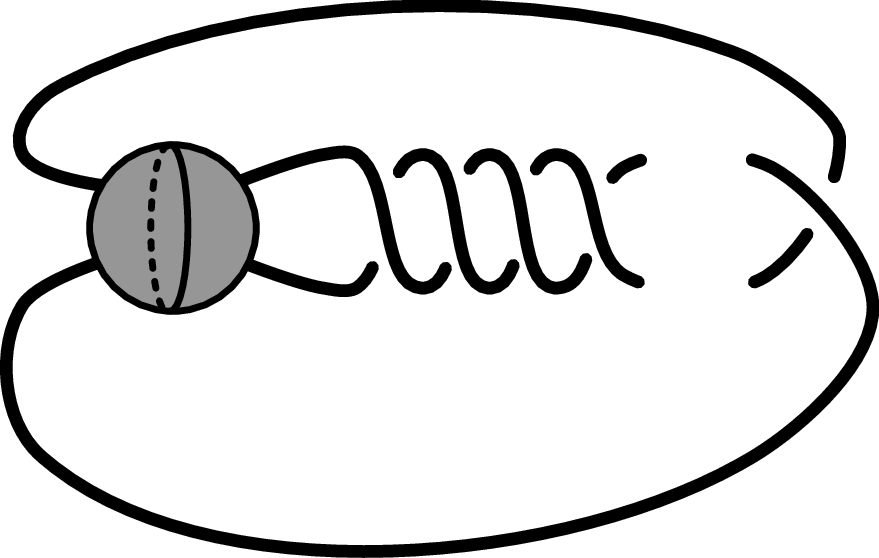}}\qquad\quad
\labellist \small 
	\pinlabel $\underbrace{\phantom{aaa}}_1\underbrace{\phantom{iaaaaaaaaa}}_3\underbrace{\phantom{aaaaaaaa}}_3$ at 351 380
\endlabellist
\raisebox{18pt}{\includegraphics[scale=0.25]{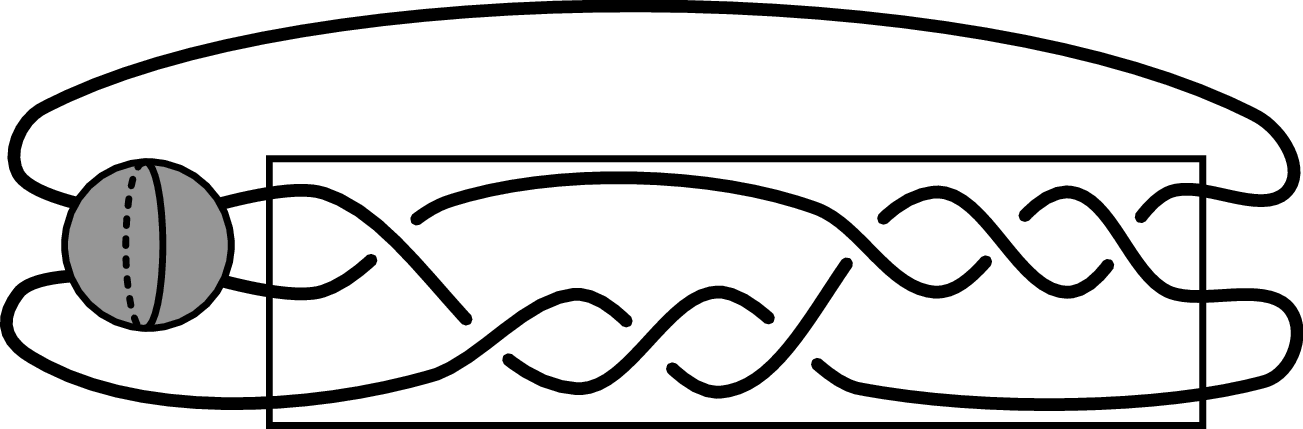}}
\end{center}
\caption{The closure $\tau(n)$ (left) giving rise to the branch set for integer surgeries (that is, Dehn fillings along slopes $n\mu+\lambda$), and the closure $\frac{13}{10}=[1,3,3]$ (right) corresponding to $13\mu+10\lambda$ Dehn filling, or $\frac{13}{10}$-surgery, in the cover.}
\label{fig:fraction}
\end{figure}

The goal of this section is to determine the tangle associated with the strong inversion on the twist knot $K_t$ illustrated in Figure \ref{fig:twist-knots}; this will establish the first part of Theorem \ref{thm:combo}. As a first approximation, we determine the tangle associated with the quotient of $S^3\smallsetminus\nu(K)$, without keeping track of the preferred representative (as described above) corresponding to the pair $\{\mu,\lambda\}$. This quotient is described in Figure \ref{fig:quotient} and is simplified in Figure \ref{fig:Kt-tangle}.

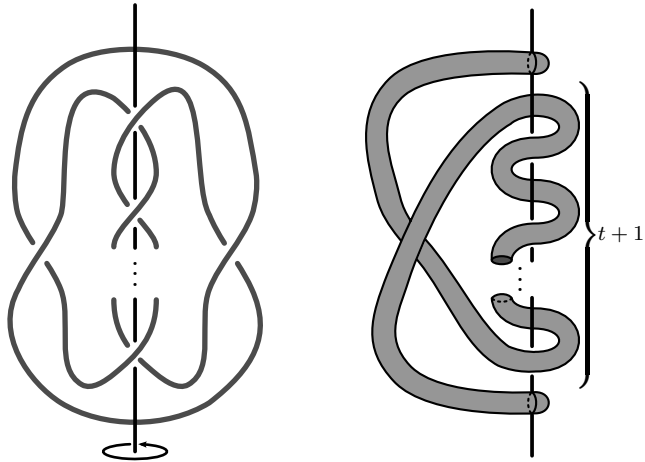
\begin{figure}[ht!]
\begin{center}
\labellist 
	\pinlabel $\vdots$ at 193 293
	\pinlabel $\vdots$ at 773 285
	\pinlabel \rotatebox{90}{$\underbrace{\phantom{aaaaaaaaaaaaaaaaaaaa}}$} at 865 340
	\pinlabel \scriptsize $t+1$ at 925 340
	\endlabellist
\raisebox{12pt}{\includegraphics[scale=0.25]{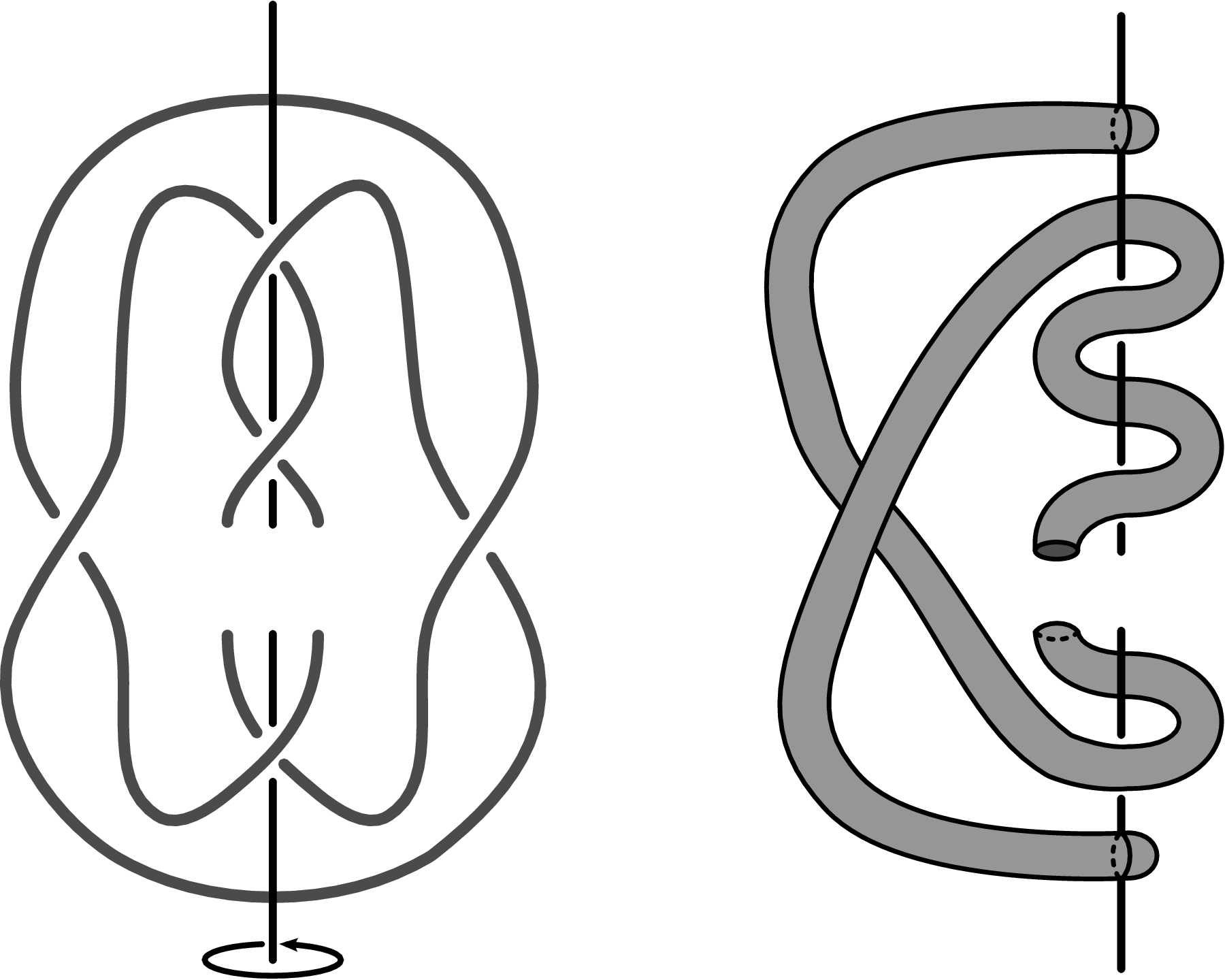}}
\end{center}
\caption{The quotient of the exterior of a strongly invertible knot determines a tangle: On the left the knot $K_t$ is shown (with axis of symmetry passing through a point at infinity), and on the right is the fundamental domain of the involution on the knot exterior. Note that the latter is always homeomorphic to a ball, and the image of the fixed point set descends to a properly embedded pair of arcs.}
\label{fig:quotient}\end{figure}

\begin{figure}[ht!]
\begin{center}
\labellist 
	\pinlabel $\vdots$ at 308 458
	\pinlabel \rotatebox{-90}{$\underbrace{\phantom{aaaaaaaaaaaaaa}}$} at 170 485
	\pinlabel \scriptsize $t+1$ at 110 485
	\tiny
	\pinlabel $a$ at 368 624
	\pinlabel $b$ at 368 565
	\pinlabel $c$ at 368 355
	\pinlabel $d$ at 368 300
\endlabellist
\qquad
\raisebox{0pt}{\includegraphics[scale=0.25]{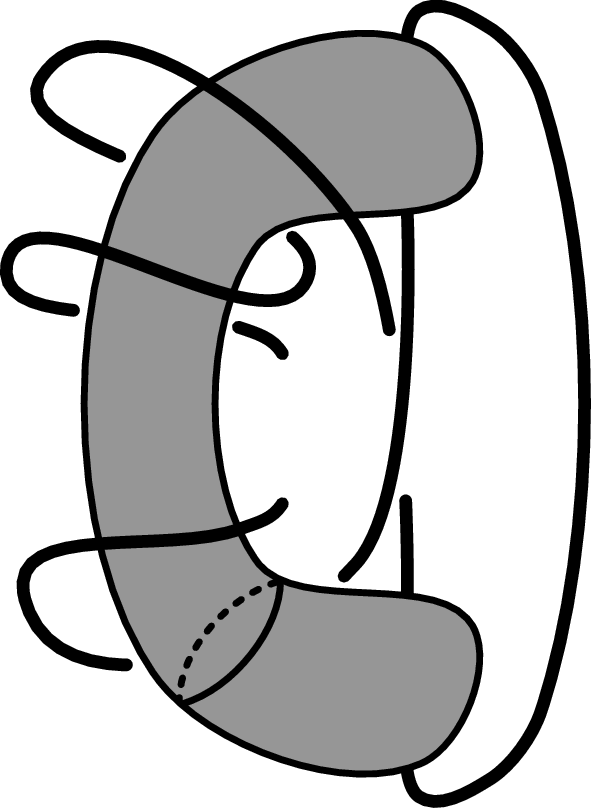}}\qquad\qquad
\labellist 
	\pinlabel $\vdots$ at 273 511
	\pinlabel \rotatebox{-90}{$\underbrace{\phantom{aaaaaaaa}}$} at 148 500
	\pinlabel \scriptsize $t-1$ at 89 500
	\tiny
	\pinlabel $a$ at 355 615
	\pinlabel $b$ at 328 590
	\pinlabel $c$ at 328 385
	\pinlabel $d$ at 355 374
\endlabellist
\raisebox{0pt}{\includegraphics[scale=0.25]{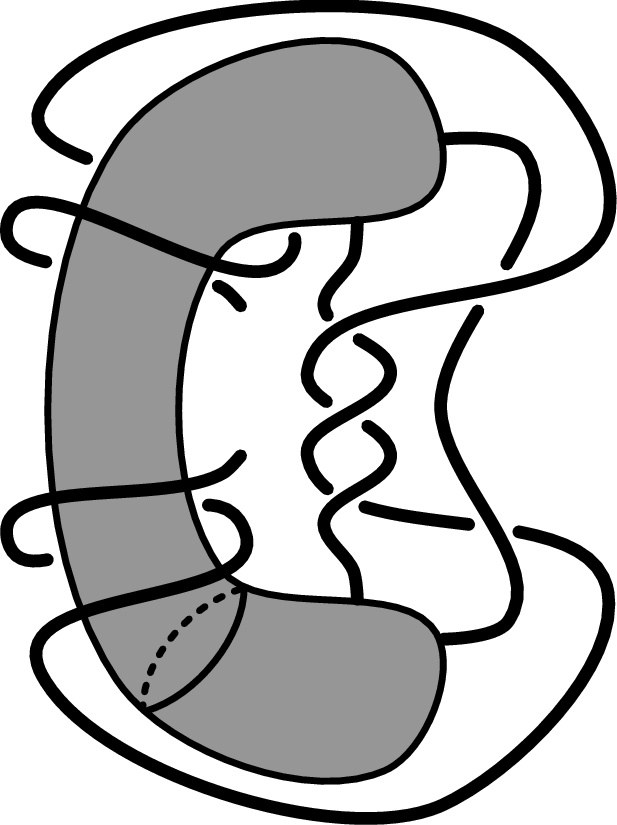}}\qquad\quad
\labellist \small 
	\pinlabel $\cdots$ at 347 411
	\pinlabel $\underbrace{\phantom{aaaaaaa}}_t$ at 342 347
	\tiny
	\pinlabel $a$ at 120 540
	\pinlabel $b$ at 120 323
	\pinlabel $c$ at 455 323
	\pinlabel $d$ at 455 540
\endlabellist	
\raisebox{12pt}{\includegraphics[scale=0.25]{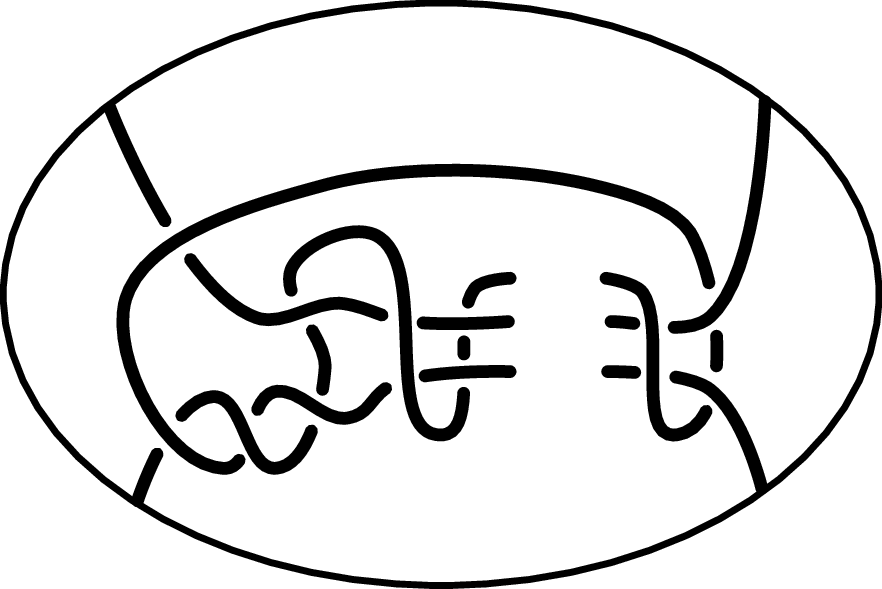}}
\end{center}
\caption{Simplifying the tangle resulting from the quotient of $K_t$ by the strong inversion. Note that the tangle diagram on the right yields the trivial knot when the endpoints labelled  $a$ and $b$ are identified and the endpoints labelled  $c$ and $d$ are identified.}
\label{fig:Kt-tangle}\end{figure}
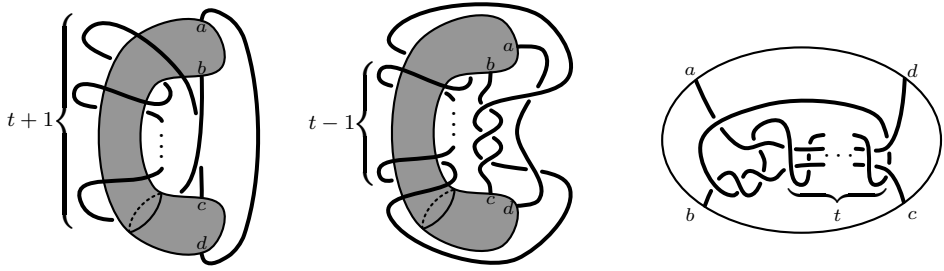

The closure of the tangle diagram in Figure \ref{fig:Kt-tangle} that gives the trivial knot (by joining $a$ to $b$ and $c$ to $d$) identifies the trivial surgery in the cover, and therefore the preferred representative $(B^3,\tau_t)$ for this tangle is obtained by adding some collection of horizontal twists. That is, the tangle diagram in Figure \ref{fig:Kt-tangle} describes a representative for the associated quotient tangle that is compatible with some integer framing, but not necessarily the Seifert framing.  In particular, the right hand half twist between the strands meeting $c$ and $d$ lifts to a full Dehn twist  along the knot meridian $\mu$ in the two-fold branched cover. 

 Recall that $\langle\si_1,\si_2 |\si_1\si_2\si_1=\si_2\si_1\si_2\rangle$ is a presentation for the three-strand braid group, where braid words are read from left to right. Our convention for the generators is given in Figure \ref{fig:braid-generators}. Figure \ref{fig:Kt-tangle} suggests that the desired collection of branch sets $\tau_t(n)$ may be represented by closures of three-strand braids (see Figure \ref{fig:twist-quotient-braid}). From this observation the relevant family of braids to consider is given by $\si_2\si_1^3\si_2(\si_2\si_1^2\si_2)^t\si_1^N$, and determining the preferred representative amounts to establishing the appropriate framing, controlled by the integer $N$.

\begin{figure}[ht!]
\begin{center}
\labellist 
	\pinlabel $\si_2=$ at 375 463
	\pinlabel $\si_1=$ at 95 463
\endlabellist
\raisebox{0pt}{\includegraphics[scale=0.4]{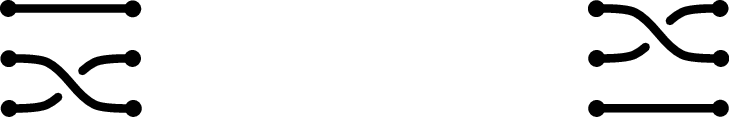}}
\end{center}
\caption{Standard generators for the three-strand braid group.}
\label{fig:braid-generators}
\end{figure}

\begin{figure}[ht!]
\begin{center}
\labellist \small 
	\pinlabel $\cdots$ at 216 142
	\pinlabel $\underbrace{\phantom{aaaaaaa}}_t$ at 210 82
\endlabellist
\raisebox{-25pt}{\includegraphics[scale=0.25]{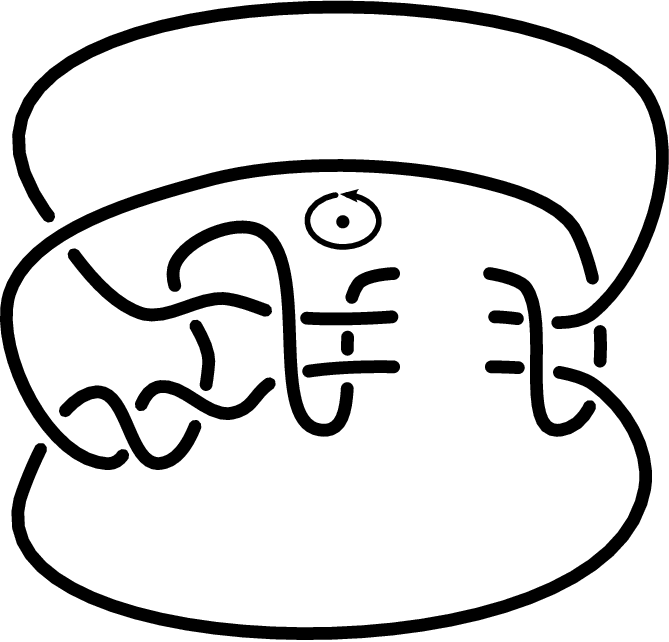}}
\qquad\qquad
\labellist 
	\pinlabel $\cdots$ at 400 415
	\pinlabel $\underbrace{\phantom{aaaaaaaaaaaaaaaaaaa}}_t$ at 397 372
\endlabellist
\raisebox{0pt}{\includegraphics[scale=0.4]{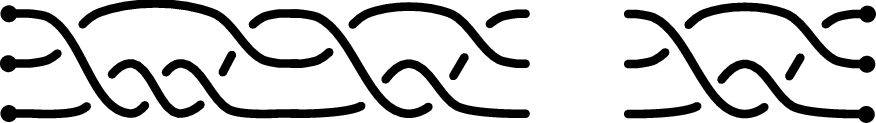}}
\end{center}
\caption{A closure of the quotient tangle associated with some (yet to be determined) integer framed surgery on $K_t$ on the left. With the braid axis indicated, this gives rise to the positive braid associated with the quotient of the knot $K_t$ on the right. Changes in framing in the two-fold branched cover (that is, Dehn twists along the meridian) are realized by the generator $\si_1$ in the base.}
\label{fig:twist-quotient-braid}
\end{figure}

We claim that the branch  set $\tau_t(n)$ is the closure of the braid \[\si_2\si_1^3\si_2(\si_2\si_1^2\si_2)^t\si_1^{N+n}\] where \[N=\begin{cases} 2+2t & {\rm for\ } t {\rm \ odd} \\ 6+2t & {\rm for\ } t {\rm \ even}\end{cases}\] so that $S^3_n(K_t)\cong \Br(S^3,\tau_t(n))$. Note that the non-negative integer $t$ determines the twist knot, while the integer $n$ is the surgery coefficient.

To see that this is the correct framing, it suffices to verify that $\det(\tau_t(0))=0$, since $|H_1(S^3_n(K_t);\bZ)|=|n|$ and $\si_1$ lifts to the positive Dehn twist about the meridian. We leave this verification to the reader (for the moment), but note that calculations in Section \ref{sec:width} will provide a proof (see Remark \ref{rmk:det}).

The braid relation establishes that $\si_2\si_1^2\si_2\si_1^2$ and $\Delta=(\si_2\si_1)^3$ are equivalent, the latter representing the full twist on three strands. As this element is central in the the braid group,  have now shown: 

\begin{proposition}\label{prp:framed-tangle}
The branch set $\tau_t(n)$ is given by the closure of the braid \[\beta_{t,n}=\begin{cases} \si_1^{2+n}(\si_2\si_1^3\si_2)\Delta^t & {\rm for\ } t {\rm \ odd} \\ \si_1^{6+n}(\si_2\si_1^3\si_2)\Delta^t & {\rm for\ } t {\rm \ even}\end{cases}\] where $\Delta=(\si_2\si_1)^3$.
\end{proposition}

\begin{remark}
It is possible to find the preferred representative directly (as in Bleiler \cite{Bleiler1985}, for example). This amounts to (carefully) carrying the image of the longitude through the quotient shown in Figure \ref{fig:quotient} and the subsequent tangle isotopies of Figure \ref{fig:Kt-tangle}. In the present setting, the key observation is that a full twist in the cover (taking $K_t$ to $K_{t+2}$) corresponds to the addition of $\Delta^2$ in the base. One checks that the addition of the full twist $\Delta^2$ in the branch set preserves the Seifert framing in the cover, and as a result it is possible to induct in $t$ (in fact, two separate inductions for $t$ odd and $t$ even) to determine that the preferred representative $(B^3,\tau_t)$. The base case $t=0$ may be extracted from \cite[Section 3.2]{Watson2008} together with the Seifert fibre structure on the trefoil exterior, while the base case $t=1$ may be extracted from \cite[Section 7.1]{Watson2008}.
\end{remark}

This calculation sets up our second task towards the proof of Theorem \ref{thm:twist}: determine the reduced Khovanov homology for closures of the family of three-braids $\beta_{t,n}$. To conclude this discussion, Jeremy Van Horn-Morris points out that realizing the branch set as the closure of a three-strand braid also establishes the following:
\begin{corollary}The result of integer Dehn surgery on $K_t$ admits a genus one open book decomposition, for each $t\ge0$. Moreover, for positive integer Dehn surgeries the corresponding open book decomposition has positive monodromy.   \end{corollary}

\section{Khovanov homology} \label{sec:Kh}

\subsection{Grading conventions for reduced Khovanov homology}\label{sub:grading} Throughout this work we make use of the reduced Khovanov homology of a link $L$, with coefficients in $\bF=\bZ/2\bZ$, denoted $\Khred(L)$ \cite{Khovanov2000,Khovanov2003}. For our purposes, this is a $\frac{1}{2}\bZ\oplus\frac{1}{2}\bZ$-graded group, with primary (cohomological) grading $\delta$ and secondary (quantum) grading $q$. Note that by setting $u=\delta+q\in\bZ$ the Jones polynomial of $L$ is given by \[\sum_{u,q}(-1)^ut^q\rk\Khred^u_q(L).\] In particular, our primary grading is by {\em diagonals} of slope 2 in the $(u,q)$-plane for the standard bigrading in Khovanov homology \cite{Khovanov2000}. 

\begin{definition}\label{def:width}
The homological width of a link $L$ is the integer $w(L)>0$ determined by the number of $\delta$-gradings supporting non-trivial homology in $\Khred(L)$.
\end{definition}

This grading convention, specifically suited to considerations involving homological width, is consistent with that of Manolescu and Ozsv\'ath \cite{MO2007}. For example, let $U$ denote the trivial knot. Then $\Khred(U)=\bF$ supported in bigrading $(0,0)$ (with $w(U)=1$), while $\Khred(U\sqcup U)=\bF\oplus\bF$, with one generator supported in each bigrading $(-\half,\half)$ and $(\half,-\half)$ (with $w(U\sqcup U)=2$). These examples are illustrated in Figure \ref{fig:small-homology}. 

\begin{figure}[ht!]
\begin{center}
\labellist
\small
	\pinlabel $1$ at 324 413
	
	\tiny
	\pinlabel $0$ at 323 378
	\pinlabel $0$ at 290 413
	\endlabellist
\raisebox{0pt}{\includegraphics[scale=0.4]{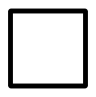}}\qquad\qquad
\labellist\small
	\pinlabel $1$ at 324 450
	\pinlabel $1$ at 360 413
	\tiny
	\pinlabel $-\frac{1}{2}$ at 318 378
	\pinlabel $\frac{1}{2}$ at 356 378
	\pinlabel $-\frac{1}{2}$ at 284 413
	\pinlabel $\frac{1}{2}$ at 290 450
	\endlabellist
\raisebox{0pt}{\includegraphics[scale=0.4]{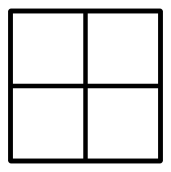}}\qquad\qquad
\labellist\small
	\pinlabel $1$ at 324 413
	\pinlabel $1$ at 324 485	
	\tiny
	\pinlabel $-\frac{1}{2}$ at 321 378
	\pinlabel $\frac{1}{2}$ at 290 413
	\pinlabel $\frac{3}{2}$ at 290 450
	\pinlabel $\frac{5}{2}$ at 290 485
\endlabellist
\raisebox{0pt}{\includegraphics[scale=0.4]{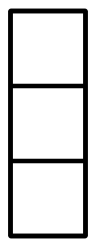}}
\end{center}
\caption{The reduced Khovanov homology of the trivial knot $U$, the two component trivial link $U\sqcup U$, and the positively clasped Hopf link $H^+$, from left to right. The primary grading is read horizontally and the secondary grading is read vertically. The entries indicate the rank of the group at the given bigrading; trivial homology in a given bigrading is left blank.}
\label{fig:small-homology}
\end{figure}

\subsection{The skein exact sequence as a mapping cone}\label{sub:mc} With the above conventions in place, the long exact sequence in Khovanov homology related to the resolution of a positive crossing may be expressed as a mapping cone
 \begin{equation}\label{eqn:cone}
 \Khred(\rightcross) \cong H_*\left( \Khred(\zero)[\textstyle-\half,\half]\to\Khred(\one)[\textstyle-\frac{1}{2}c,\frac{1}{2}(3c+2)]\right)\end{equation}
 where the connecting homomorphism raises the primary grading by one and fixes the secondary grading. The shift in bigrading is defined by \[\Khred^\delta_q(L)[i,j]=\Khred^{\delta-i}_{q-j}(L).\] The constant $c$ measures the difference in negative crossings $n_-(\one)-n_-(\rightcross)$ for some choice of orientation on the components of the resolution $\one$ that do not inherit an orientation from $\rightcross$. Notice that when $L$ is the closure of a positive braid $c=n_-(\one)$. These conventions are consistent with those of Rasmussen \cite{Rasmussen2005} and of Manolescu and Ozsv\'ath \cite{MO2007}.

\subsection{Spectral sequences from iterated mapping cones}\label{sub:iterate} There are various settings in which the mapping cone described above may be efficiently iterated to calculate Khovanov homology. For the present purposes, we establish one such situation given a link represented as the closure of a positive braid. To this end, let $L=\close{\beta}$ where $\beta$ is a positive braid, and fix $n$ distinguished crossings. Fixing, additionally, an order on the distinguished crossings gives rise to a collection of positive braids $\beta_i$ obtained by replacing the first $i$ crossings $\rightcross$ with the oriented resolution $\zero$. Similarly, the link $\close{\beta}$ produces a collection of resolutions $R_i$ obtained by replacing the first $i-1$ crossings $\rightcross$ by the oriented resolution $\zero$ and the $i^{\rm th}$ with unoriented resolution $\one$. (A schematic illustrating the case $n=2$ in an example is shown in Figure \ref{fig:torus-scheme}.) Of course, the $R_i$ are no longer link diagrams in closed-braid form. For each link $R_i$ fix the constant $c_i=n_-(R_i)$ for some choice of orientation on $R_i$ compatible with the orientation on the unaffected components of $\close{\beta_{i}}$, as in the previous section. As a result, from the grading shifts in (\ref{eqn:cone}) we have that 
\begin{equation}\label{eqn:braid-cone}\Khred(\close{\beta_{i-1}}) \cong H_*\left( \Khred(\close{\beta_{i}})[\textstyle-\half,\half]\to\Khred(R_i)[\textstyle-\frac{1}{2}c_i,\frac{1}{2}(3c_i+2)]\right)
\end{equation}
 for $1\leq i\leq n$, where $\beta=\beta_0$. Now define 
 \begin{equation}\label{eqn:braid-SS}\begin{split}
 \sC_i= \begin{cases} \CKhred(R_i)[\textstyle-\frac{1}{2}(c_i+i-1),\frac{1}{2}(3c_i+2+i-1)] & 1\leq i\leq n \\ \CKhred(\close{\beta_n})[-\textstyle\frac{n}{2},\frac{n}{2}] & i=n+1. \end{cases}
\end{split}\end{equation} 
Consulting the definition of $\CKhred(L)$, we see that $\CKhred(L)\cong \bigoplus_{i=1}^{n+1}\sC_i$ as a bigraded $\bF$-vector space. Moreover, omitting the internal differentials on the $\sC_i$ for brevity, the structure of the differential relative to this decomposition is given by
\[\xymatrix@C=40pt@R=5pt{
	\sC_1
	&&& \\
	&
	\sC_2\ar@/^0.7pc/[ul]
	&& \\
	&&
	\sC_3\ar@/^0.7pc/[ul]\ar@/^2pc/[uull]
	&
	\sC_4\ar@/_0.7pc/[l]\ar@/_0.9pc/[ull]\ar@/_1pc/[uulll]\\
	}\] 
(in the case $n=3$) so that, by construction, this decomposition of the chain complex comes with a filtration of the form $\bigoplus_{j=1}^i\sC_j\subset\bigoplus_{j=1}^{i+1}\sC_{j}$ for $1\leq i\leq n$. (The differential takes the form $\sum_{i\ge j} \partial_{i,j}$ in general, where $\partial_{i,j}\co \sC_i\to\sC_j$.) As a result, there is an associated spectral sequence with $E_0(\beta)\cong\bigoplus_{i=1}^{n+1}\sC_i$, converging to $E_\infty(\beta)\cong\Khred(L)$. We will be most interested in the $E_1$ page of this spectral sequence given by 
\begin{equation}\label{eqn:braid-E1}E_1(\beta)\cong\Khred(\close{\beta_n})[-\textstyle\frac{n}{2},\frac{n}{2}]\oplus\left(\bigoplus_{i=1}^n\Khred(R_i)[\textstyle-\frac{1}{2}(c_i-1+i),\frac{1}{2}(3c_i+1+i)]\right).\end{equation}This notation does not reference the order chosen on the resolved crossings, however the crossings and their order will always be made explicit in applications. Notice that the differential on the $E_1$ page is lower triangular (and filtered by diagonals) if we choose a basis corresponding to \[\Khred(\close{\beta_n})\oplus\Khred(R_n)\oplus\Khred(R_{n-1})\oplus\cdots\oplus\Khred(R_1)\] (the grading shifts have been omitted for brevity). In particular, $\Khred(\close{\beta_n})[-\textstyle\frac{n}{2},\frac{n}{2}]$ is never in the target of the differential. 

While the higher differentials may be difficult to calculate in general, the construction ensures that each of these raises $\delta$-grading by one and fixes the $q$-grading, as in the case of the mapping cone from which this spectral sequence is derived. (Note that we are ignoring the spectral sequence grading and working with the grading inherited from $\CKhred(L)$.) Indeed, this construction in the case $n=1$ gives the two step complex 
\[\Khred(\close{\beta_1})[\textstyle-\half,\half]\to\Khred(R_1)[\textstyle-\frac{1}{2}c_1,\frac{1}{2}(3c_1+2)]\] associated with the mapping cone, as in the previous section. 

Even in the case $n=1$, obtaining the differential essentially amounts to calculating the reduced Khovanov homology directly from the definition. In practice, computing the homology by way of this spectral sequence is tantamount to $n$ (careful) iterations of the long exact sequence. However, the $E_1$ page described here can be useful when combined with additional structure (described below). It should be noted that this spectral sequence is closely related to the spectral sequence used by Turner \cite{Turner2008} to compute the (unreduced) Khovanov homology of $(3,q)$-torus links (with coefficients in $\bQ$). Indeed, we will revisit this calculation in Section \ref{sub:twist}. However, we reiterate that the set-up here ignores the spectral sequence gradings; we opt instead to use the grading inherited from $\CKhred$. Another related instance of this construction is given in \cite[Lemma 4.10]{Watson2008} (see also Proposition \ref{prp:iteration}).

\subsection{Turner's spectral sequence}\label{sub:turner} There is an analogy to Lee's spectral sequence \cite{Lee2005} for reduced Khovanov homology due to Turner \cite{Turner2006}, summarized as follows:  

\begin{theorem}[Turner \cite{Turner2006}]\label{thm:turner} There is a perturbed version of the reduced Khovanov complex with homology denoted $\Khred'(L)$ for which $\rk\Khred'(L)=2^{|L|-1}$. Moreover, there is a spectral sequence converging to $E_\infty'=\Khred'(L)$, with $E_1'=\Khred(L)$, satisfying the properties that \\
(1) the differential $d_i$ on the $E_i'$ page is of bi-degree $(1-i,i)$, and \\
(2) for each $n\in\bZ$  \begin{equation}\rk\bigoplus_{\delta+q=n}\Khred'^\delta_q(L)=\frac{1}{2}\displaystyle\Big|\big\{X\subset\{1,\ldots,|L|\}\ \co\delta+q=2\sum_{l\in X, m\notin X}\operatorname{lk}(L_l,L_m)\big\}\Big|\end{equation} where the $L_i$ denote the components of the link $L$; therefore $\rk\bigoplus_{\delta+q=n}\Khred^\delta_q(L)$ is at least this quantity. \end{theorem}


Our main computational tool will be in combining the structure of $\Khred'(L)$ with the iterated the mapping cone. Since these both arise from spectral sequence constructions, we will endeavour to refer to {\em pairings} when higher differentials computing $\Khred'(L)$ are in question, and {\em differentials} when referring to the iterated mapping cone in an attempt to avoid confusion. 
 
\begin{remark}Note that our usage of the label {\em Turner's spectral sequence} differs from certain instances of the same term in the literature. Our usage refers to the perturbation described in Theorem \ref{thm:turner} (the subject of \cite{Turner2006}) and not the spectral sequence defined and exploited in \cite{Turner2008}. \end{remark}

\subsection{The reduced Khovanov homology of a full twist}\label{sub:twist} The reduced Khovanov homology for the first 3 non-trivial torus links on three strands is described in Figure \ref{fig:base-cases}. The $\delta$-grading is read horizontally and the $q$-grading is read vertically. The values at a given lattice point denote the rank of the $\bF$-vector space in that bigrading; empty lattice points are read as rank zero.  

\begin{figure}[ht!]
\begin{center}
\labellist\small
	\pinlabel $1$ at 324 413
	\pinlabel $1$ at 324 485
	\pinlabel $1$ at 324 522
	
	\tiny
	\pinlabel $-1$ at 321 378
	\pinlabel $1$ at 290 413
	\pinlabel $2$ at 290 450
	\pinlabel $3$ at 290 485
	\pinlabel $4$ at 290 522
	\endlabellist
\raisebox{0pt}{\includegraphics[scale=0.4]{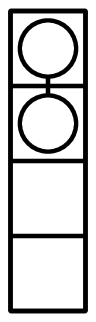}}\qquad\qquad
\labellist\small
	\pinlabel $1$ at 324 413
	\pinlabel $1$ at 324 485
	\pinlabel $1$ at 324 522
	\pinlabel $2$ at 324 557
	
	\pinlabel $1$ at 359 522
	
	\tiny
	\pinlabel $-2$ at 321 378
	\pinlabel $-1$ at 356 378
	\pinlabel $2$ at 290 413
	\pinlabel $3$ at 290 450
	\pinlabel $4$ at 290 485
	\pinlabel $5$ at 290 522
	\pinlabel $6$ at 290 557
	
	\endlabellist
\raisebox{0pt}{\includegraphics[scale=0.4]{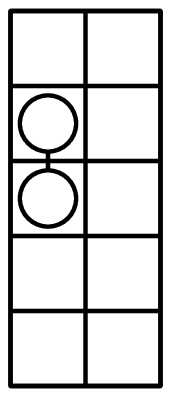}}\qquad\qquad
\labellist\small
	\pinlabel $1$ at 324 413
	\pinlabel $1$ at 324 485
	\pinlabel $1$ at 324 522
	\pinlabel $1$ at 324 594
	
	\pinlabel $1$ at 359 522
	
	\tiny
	\pinlabel $-3$ at 321 378
	\pinlabel $-2$ at 356 378
	\pinlabel $3$ at 290 413
	\pinlabel $4$ at 290 450
	\pinlabel $5$ at 290 485
	\pinlabel $6$ at 290 522
	\pinlabel $7$ at 290 557
	\pinlabel $8$ at 290 594
\endlabellist
\raisebox{0pt}{\includegraphics[scale=0.4]{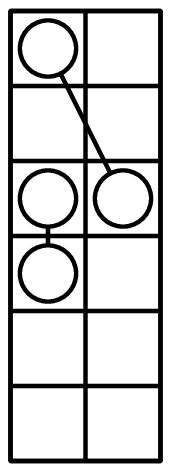}}
\end{center}
\caption{The Khovanov homology of the $(3,2)$-, $(3,3)$- and $(3,4)$-torus links, from left to right. Note that the generators have been paired in each case illustrating the higher differentials present in Turner's spectral sequence.}
\label{fig:base-cases}
\end{figure}

Using Lee's spectral sequence, Turner establishes the full Khovanov homology (with coefficients in $\bQ$) for three-strand torus links \cite{Turner2008}. Following this proof, a similar result may be established in the present setting, modulo a particular indeterminate summand described in Figure \ref{fig:tetris}. We will ultimately be interested in the $(3,3t)$-torus links $T_{3,3t}$ for $t>0$. Note that this is the closure of the braid $(\si_2\si_1)^{3t}=\Delta^t$. 

\begin{figure}[ht!]
\begin{center}
\raisebox{0pt}{\includegraphics[scale=0.4]{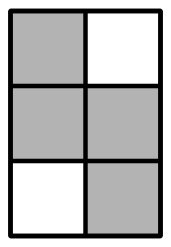}}\qquad\qquad
\labellist\small
	\pinlabel $1$ at 324 450
	\pinlabel $1$ at 324 485
	
	\pinlabel $1$ at 360 413
	\pinlabel $1$ at 360 450
	\endlabellist
\raisebox{0pt}{\includegraphics[scale=0.4]{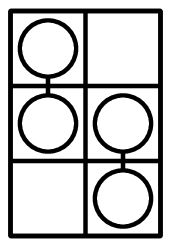}}\qquad\qquad
\labellist\small
	\pinlabel $1$ at 324 485
	\pinlabel $1$ at 360 413
\endlabellist
\raisebox{0pt}{\includegraphics[scale=0.4]{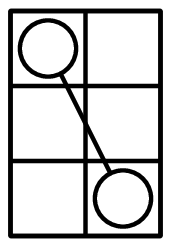}}
\end{center}
\caption{The two possibilities for the indeterminate summand, with pairings labelled corresponding to higher differentials in Turner's spectral sequence.}
\label{fig:tetris}
\end{figure}

\begin{remark}\label{rmk:tetris}  Since we will be interested only in coarse properties of the Khovanov  homology of the full twist in the sequel (in particular, the homological width), the actual value of this indeterminate summand will have no contribution. In practice, the actual rank of this summand is 2 (as in $\Khred(T_{3,4})$ for example, see Figure \ref{fig:base-cases}), though we will not prove this in general. Calculation up to indeterminate summands is relatively easy, and in general the homological width of a link seems easier to determine than the complete Khovanov homology.   \end{remark}

\begin{proposition}\label{prp:torus} The reduced Khovanov homology of the positive $T_{3,q}$ torus link is described in  Figure \ref{fig:Kh-torus-knots} for $q>0$.
\end{proposition}

\begin{figure}[ht!]
\begin{center}
\ \ \
\labellist\small
	\pinlabel $1$ at 324 413
	\pinlabel $1$ at 324 485
	\pinlabel $1$ at 324 522
	
	\pinlabel $1$ at 360 594
	\pinlabel $1$ at 360 630
		
	\pinlabel $1$ at 432 702	
	\pinlabel $1$ at 432 739	
		
	\pinlabel $1$ at 468 810	
	\pinlabel $1$ at 468 846

	\pinlabel \rotatebox{-45}{$\vdots$} at 400 671

	\tiny
	\pinlabel \rotatebox{-90}{$2-3t$} at 322 363
	\pinlabel \rotatebox{-90}{$1-3t$} at 357 363
	\pinlabel $\hdots$ at 396 365
	\pinlabel \rotatebox{-90}{$1-2t$} at 468 363
	
	\pinlabel $3t-2$ at 275 413
	\pinlabel $3t-1$ at 275 450
	\pinlabel $\vdots$ at 285 495
	
	\pinlabel $6t-3$ at 275 810
	\pinlabel $6t-2$ at 275 845		
	\endlabellist
\raisebox{0pt}{\includegraphics[scale=0.4]{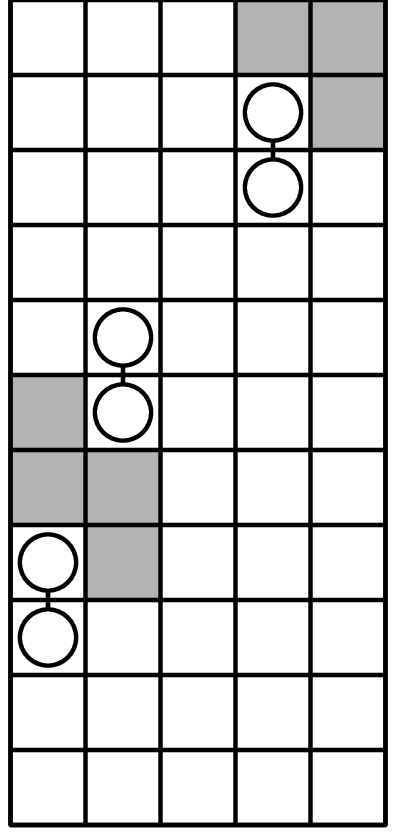}}
\quad\qquad
\labellist\small
	\pinlabel $1$ at 324 413
	\pinlabel $1$ at 324 485
	\pinlabel $1$ at 324 522
	
	\pinlabel $1$ at 360 594
	\pinlabel $1$ at 360 630
		
	\pinlabel $1$ at 432 702	
	\pinlabel $1$ at 432 739	
		
	\pinlabel $1$ at 468 810	
	\pinlabel $1$ at 468 846

	\pinlabel $2$ at 468 879
	
	\pinlabel $1$ at 504 844

	\pinlabel \rotatebox{-45}{$\vdots$} at 400 671

	\tiny
	\pinlabel \rotatebox{-90}{$1-3t$} at 322 363
	\pinlabel \rotatebox{-90}{$2-3t$} at 357 363
	\pinlabel $\hdots$ at 396 365
	\pinlabel \rotatebox{-90}{$1-2t$} at 503 363
	
	\pinlabel $3t-1$ at 275 413
	\pinlabel $3t$ at 285 450
	\pinlabel $\vdots$ at 285 495
	
	\pinlabel $6t-1$ at 275 845
	\pinlabel $6t$ at 285 881
	
	\endlabellist
\raisebox{0pt}{\includegraphics[scale=0.4]{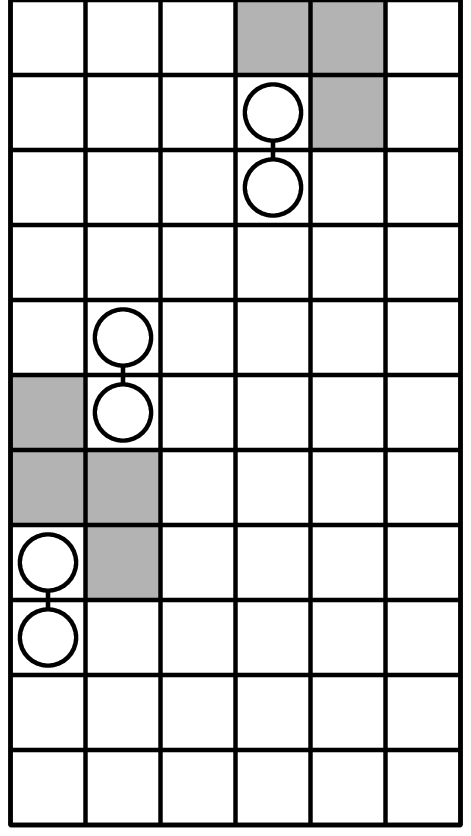}}
\quad\qquad
\labellist\small
	\pinlabel $1$ at 324 413
	\pinlabel $1$ at 324 485
	\pinlabel $1$ at 324 522
	
	\pinlabel $1$ at 360 594
	\pinlabel $1$ at 360 630
		
	\pinlabel $1$ at 432 702	
	\pinlabel $1$ at 432 739	
		
	\pinlabel $1$ at 468 810	
	\pinlabel $1$ at 468 846

	\pinlabel \rotatebox{90}{$\underbrace{\phantom{iaaaaaaaaaaaaaaaaaaaaaaaaaaaaaaaaaa}}$} at 525 720
	\pinlabel $\times t$ at 550 720
	\pinlabel \rotatebox{-45}{$\vdots$} at 400 671

	\tiny
	\pinlabel \rotatebox{-90}{$-3t$} at 322 367
	\pinlabel \rotatebox{-90}{$1-3t$} at 357 363
	\pinlabel $\hdots$ at 396 365
	\pinlabel \rotatebox{-90}{$-2t$} at 505 367
	
	\pinlabel $3t$ at 285 413
	\pinlabel $3t+1$ at 275 450
	\pinlabel $\vdots$ at 285 495
	
	\pinlabel $6t+1$ at 275 881
	\pinlabel $6t+2$ at 275 918

	\endlabellist
\raisebox{0pt}{\includegraphics[scale=0.4]{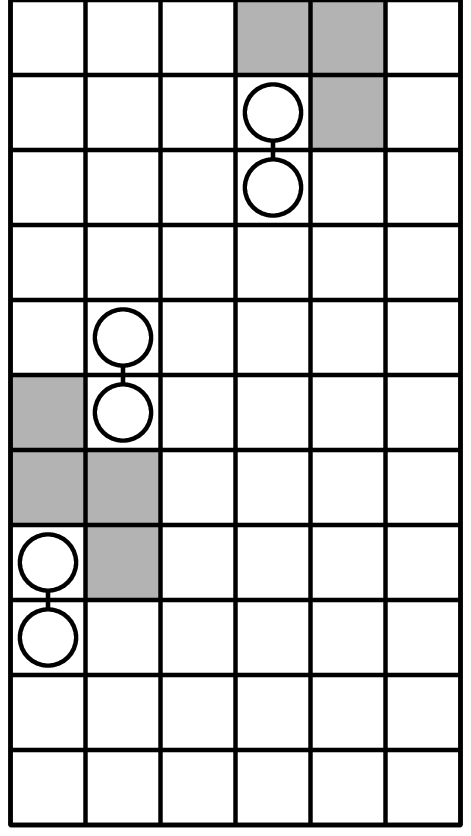}}
\vspace*{20pt}
\end{center}
\caption{The reduced Khovanov homology of the torus links $T_{3,3t-1}$, $T_{3,3t}$ and $T_{3,3t+1}$, for $t>0$, from left to right. Note that $\Khred(T_{3,3t-1})$ and $\Khred(T_{3,3t})$ have only $t-1$ indeterminate summands. }
\label{fig:Kh-torus-knots}
\end{figure}

Following Turner \cite{Turner2008}, the proof of Proposition \ref{prp:torus} is by induction in $t$ (with base case $t=1$ provided by Figure \ref{fig:base-cases}) in three steps: 

\begin{claim}\label{cl:1}If the result holds for $T_{3,3t-1}$ then it holds for $T_{3,3t}$.\end{claim}

\begin{claim}\label{cl:2}If the result holds for $T_{3,3t}$ then it holds for $T_{3,3t+1}$.\end{claim}

\begin{claim}\label{cl:3}If the result holds for $T_{3,3t+1}$ then it holds for $T_{3,3t+2}=T_{3,3(t+1)-1}$.\end{claim}

The strategy to prove each claim is identical: for appropriately chosen $q$, iterate the skein exact sequence as summarized in Figure \ref{fig:torus-scheme} to obtain  $E_1((\si_2\si_1)^{q+1})$ as described in Section \ref{sub:iterate}. In each case this expresses $\Khred(T_{3,q+1})$ in terms of $\Khred(T_{3,q})$ together with a collection of possible new generators. With these possible new generators in hand, the structure of $\Khred'(T_{3,q+1})$ forces the behaviour of the differentials to give the result in each case. 

\begin{figure}[ht!]
\begin{center}
\xymatrix@C=65pt@R=0pt{
	&
	{\raisebox{0pt}{\includegraphics[scale=0.4]{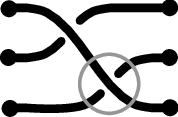}}}\ar@/^1pc/[dr] 
	&&\\
	{\labellist\small
	\pinlabel $R_1$ at 330 431
	\endlabellist
	\raisebox{0pt}{\includegraphics[scale=0.4]{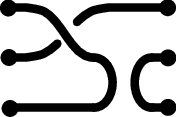}}}\ar@/^1pc/[ur]^-{c_1}
	&&
	{\raisebox{0pt}{\includegraphics[scale=0.4]{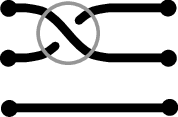}}}\ar@/^1pc/[dr]
	 &\\
	&
	{\labellist\small
	\pinlabel $R_2$ at 323 431
	\endlabellist\raisebox{0pt}{\includegraphics[scale=0.4]{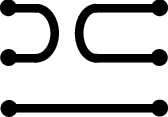}}}\ar@/^1pc/[ur]^-{c_2}
	&&
	{\raisebox{0pt}{\includegraphics[scale=0.4]{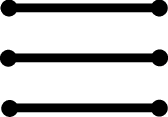}}}
	& \\
	}
\end{center}
\caption{Schematic for proving the claims: the constants $c_1$ and $c_2$ count the number of negative crossings -- relevant to the grading shifts -- in the unoriented resolution $R_1$ and $R_2$ at each step to construct $E_1((\si_2\si_1)^{q+1})$.}
\label{fig:torus-scheme}
\end{figure}

\begin{figure}[ht!]
\begin{center}
\labellist\small
	\pinlabel $1$ at 324 413
	
	\pinlabel $1$ at 359 484
	\pinlabel $1$ at 359 522
	
	\pinlabel $2$ at 359 557
	
	\pinlabel $1$ at 396 522
	
	\tiny
	\pinlabel $\cdots$ at 330 370
	\pinlabel \rotatebox{-90}{$-2t$} at 359 375
	\pinlabel \rotatebox{-90}{$1-2t$} at 390 365
	
	\pinlabel $\vdots$ at 290 530
	\pinlabel $6t$ at 290 557
	\endlabellist
\raisebox{0pt}{\includegraphics[scale=0.4]{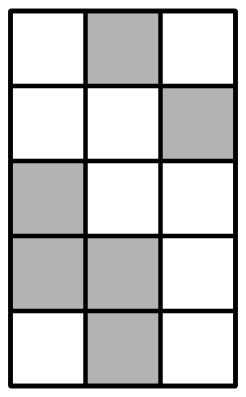}}
\qquad\qquad\qquad
\labellist\small
	\pinlabel $1$ at 324 413
	
	\pinlabel $1$ at 359 484
	\pinlabel $1$ at 359 522
	
	\pinlabel $2$ at 359 557
	
	\pinlabel $1$ at 396 522
	
	\pinlabel $2$ at 396 557
	
	\pinlabel $1$ at 359 595
	
	\tiny
	\pinlabel $\cdots$ at 330 370
	\pinlabel \rotatebox{-90}{$-1-2t$} at 359 357
	\pinlabel \rotatebox{-90}{$-2t$} at 390 370
	
	\pinlabel $\vdots$ at 275 530
	\pinlabel $6t+1$ at 275 557
	\pinlabel $6t+2$ at 275 593
	\endlabellist
\raisebox{0pt}{\includegraphics[scale=0.4]{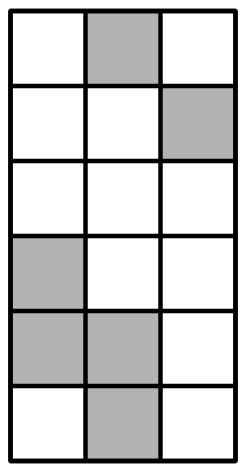}}
\qquad\qquad\qquad
\labellist\small
	\pinlabel $1$ at 324 413
	
	\pinlabel $1$ at 359 484
	\pinlabel $1$ at 359 522
	
	\pinlabel $1$ at 395 593
	\pinlabel $1$ at 395 630
	
	\tiny
	\pinlabel $\cdots$ at 330 370
	\pinlabel \rotatebox{-90}{$-2-2t$} at 359 357
	\pinlabel \rotatebox{-90}{$-1-2t$} at 390 357
	
	\pinlabel $\vdots$ at 275 530
	\pinlabel $6t+2$ at 275 557
	\pinlabel $6t+3$ at 275 593
	\pinlabel $6t+4$ at 275 630

	\endlabellist
\raisebox{0pt}{\includegraphics[scale=0.4]{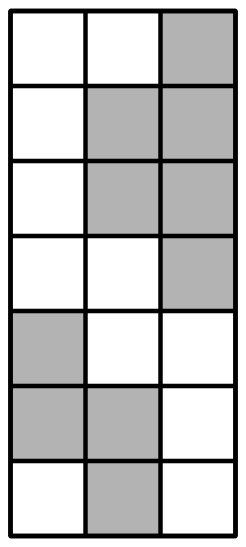}}
\end{center}
\vspace*{20pt}
\caption{Proving Claims \ref{cl:1}, \ref{cl:2} and \ref{cl:3} from left to right: two iterations of the long exact sequence provide a collection of possible new generators (these have been shaded in each case) contributing to $E_1((\si_2\si_1)^{q+1})$ calculating $\Khred(T_{3,q+1})$. Shaded but otherwise unmarked lattice points correspond to indeterminate summands, as usual.}
\label{fig:Kh-torus-claims}
\end{figure}

\begin{proof}[Proof of Claim \ref{cl:1}]
For $t>1$ we have that $c_1=c_2=4t-1$ where the unoriented resolutions are the links $R_1\simeq U\sqcup U$ and $R_2\simeq U$. As a result, the relevant grading shift on $\Khred(R_i)$ from (\ref{eqn:braid-SS}) is 
\[\textstyle[-\frac{1}{2}(c_i+i-1),\frac{1}{2}(3c_i+2+i-1)]=[-2t+1-\frac{i}{2},6t-1+\frac{i}{2}]\]
for $i=1,2$. Now $\Khred(T_{3,3t})$  is computed from $E_1((\si_2\si_1)^{3t})$, which according to (\ref{eqn:braid-E1}) is described by 
\begin{equation*}
 \textstyle \Khred(T_{3,3t-1})[-1,1]\oplus\left(\Khred(U)[-2t,6t] \oplus \Khred(U\sqcup U)[-2t+\half,6t-\half]\right).
\end{equation*}
Hence we need only consider 3 possible new generators added to $\Khred(T_{3,3t-1})[-1,1]$ in order to calculate $\Khred(T_{3,3t})$; the relevant portion of this calculation is illustrated in Figure \ref{fig:Kh-torus-claims}. Since  the possible differentials raise $\delta$-grading by 1 and fix the $q$-grading, there is only one possible non-trivial differential to consider. 

Now notice that $\rk\Khred'(T_{3,3t})=4$ as $T_{3,3t}$ is a three component link. Moreover, applying Theorem \ref{thm:turner} we see that the the 3 new generators added to $\delta+q=4t$ must survive in $\Khred'(T_{3,3t})$ (along with the generator in $\delta+q=0$). Since the remaining generators admit a unique pairing corresponding to the higher differentials calculating $\Khred'(T_{3,3t})$, we conclude that the possible differential under consideration must be zero, and 
\begin{align*}\textstyle\Khred(T_{3,3t}) 
 & \cong \textstyle \Khred(T_{3,3t-1})[-1,1]\oplus\Khred(U)[-2t,6t] \oplus \Khred(U\sqcup U)[-2t+\half,6t-\half]
\end{align*}
establishing the desired result.
\end{proof}

\begin{proof}[Proof of Claim \ref{cl:2}]
In this case we have that $c_i=4t$ and the unoriented resolutions are $R_1\simeq U$ and $R_2\simeq U\sqcup U$ respectively. The relevant grading shift on $\Khred(R_i)$ from (\ref{eqn:braid-SS}) is 
\[\textstyle[-\frac{1}{2}(c_i+i-1),\frac{1}{2}(3c_i+2+i-1)]=[-2t+\half-\frac{i}{2},6t+\half+\frac{i}{2}]\] for $i=1,2$.
As a result $\Khred(T_{3,3t+1})$  is computed from $E_1((\si_2\si_1)^{3t+1})$, which according to (\ref{eqn:braid-E1}) is described by
\begin{equation*}
 \textstyle \Khred(T_{3,3t})[-1,1]\oplus\left(\Khred(U\sqcup U)[-2t-\half,6t+\frac{3}{2}]\oplus \Khred(U)[-2t,6t+1]\right).
\end{equation*}
The relevant portion of this calculation, again featuring 3 new possible generators, is shown in Figure \ref{fig:Kh-torus-claims}. In this case, the differentials are necessarily non-trivial. Indeed, since $\Khred'(T_{3,3t+1})\cong\bF$ supported in $\delta+q=0$, all other generators of $\Khred(T_{3,3t+1})$ (in particular, those illustrated in Figure \ref{fig:Kh-torus-claims}) must pair according to higher differentials computing $\Khred'(T_{3,3t+1})$. As a result, there must be a non-trivial differential corresponding to $\bF^2\to \bF^2$ in Figure \ref{fig:Kh-torus-claims}, however this is precisely the setting wherein the indeterminate summand plays a role. That is, we note that there are two possibilities for such a pairing (as in Figure \ref{fig:tetris}), and hence two possibilities for the rank of the differential in question. Regardless of which occurs however, we obtain the result as claimed. 
\end{proof}

\begin{proof}[Proof of Claim \ref{cl:3}] Here $c_1=4t+2$ while $c_2=4t+1$, however the resolutions $R_i$ are trivial knots in each case. As a result we have only two new possible generators to consider in computing $\Khred(T_{3,3t+2})$ using $E_1((\si_2\si_1)^{3t+2})$:
\begin{equation*}
 \textstyle \Khred(T_{3,3t+1})[-1,1]\oplus\left(\Khred(U)[-2t-1,6t+3]\oplus \Khred(U)[-2t-1,6t+4]\right).
\end{equation*}
This is again summarized in Figure \ref{fig:Kh-torus-claims}. Notice that $\Khred'(T_{3,3t+2})\cong\bF$ supported in $\delta+q=0$. As a result, the new generators must cancel in $\Khred'(T_{3,3t+2})$, and the only way this can occur is for the new generators to pair with each other. As a result, the only possible non-trivial differential in this case must be zero, so that $\Khred(T_{3,3t+2})\cong E_1((\si_2\si_1)^{3t+2})$ described by 
\begin{equation*} \textstyle \Khred(T_{3,3t+1})[-1,1]\oplus\Khred(U)[-2t-1,6t+3]\oplus \Khred(U)[-2t-1,6t+4].\end{equation*} Consulting Figure \ref{fig:Kh-torus-knots} (replacing $t$ with $t+1$ so that $T_{3,3(t+1)-1}=T_{3,3t+2}$), this is the desired result. 
\end{proof}

\section{Homological width}\label{sec:width}

To complete the proof of Theorem \ref{thm:twist} it remains to establish the second part of Theorem \ref{thm:combo}. 

\begin{definition}Given a strong inversion on a knot $K$, and preferred representative for the associated quotient tangle $(B^3,\tau)$, define $w_K=\min_{r\in\bQ}\left\{w(\tau(r))\right\}$. \end{definition}

It is established in \cite[Section 4]{Watson2008} that this value is both well-defined and calculable.  Our goal then is to determine $w_{K_t}$ for the twist knots $K_t$ described in Figure \ref{fig:twist-knots}.

\begin{proposition}\label{prp:branch-width} $w_{K_t}= t+1$ for all $t\ge0$.
\end{proposition}

In the present setting $w_{K_t}$ is completely determined by the reduced Khovanov homology of two branch sets corresponding to integer surgeries (for any $t\ge0$). To see this, let \[\ell=\begin{cases}-1 & {\rm for\ } t {\rm \ odd} \\ -5 & {\rm for\ } t {\rm \ even,}\end{cases}\]  and consider the particular branch set $\tau_t(\ell)$ (so that $S^3_\ell(K_t)\cong\Br(S^3,\tau_t(\ell))$). Before turning to the proof of Proposition \ref{prp:branch-width}, we will need to calculate $\Khred(\tau_t(\ell))$. As in Section \ref{sub:twist}, $\Khred(\tau_t(\ell))$ will be expressed up to indeterminate summands. 

\subsection{$\ell$-framed surgery on twist knots}
The following proposition gives a useful special case of the iterated mapping cone construction of Section \ref{sub:twist}.  

\begin{proposition}\label{prp:iteration} Let $\beta=\beta_1\sigma_i^n\beta_2$ be a positive braid.  According to (\ref{eqn:braid-E1}), by applying the iterated mapping cone to the crossings $\sigma_i^n$,  $\Khred(\overline{\beta})$ may be computed by a spectral sequence with 
\[E_1(\beta_1\sigma_i^n\beta_2)\cong \Big(\Khred(\close{\beta_1\beta_2})\oplus\big(\bigoplus_{q=0}^{n-1}\Khred(R)[\textstyle-\half(c-1),\half(3c+1+2q)]\big)\Big)[\textstyle-\frac{n}{2},\frac{n}{2}]\] 
where $R$ is the link obtained by replacing the entire twist region corresponding to $\si_i^n$ with the unoriented resolution $\one$ and $c=n_-(R)$. In particular, the links $R_i$ of Section \ref{sub:iterate} differ from $R$ by a series of Reidemeister 1 moves and $c_i=c+n-i$. \end{proposition}

As before, this provides a candidate collection of generators that may be used to determine the homology by making use of the additional structure of Theorem \ref{thm:turner}. Now notice that the knot $\tau_t(\ell)$ is  the closure of the braid $\si_1\si_2\si_1^3\si_2\Delta^t$ (see Proposition \ref{prp:framed-tangle}). Using the result in Proposition \ref{prp:torus}, and a similar strategy of proof employing Proposition \ref{prp:iteration}, it is possible to determine the reduced Khovanov homology, up to indeterminate summands, for the branch set $\tau_t(\ell)$. 

\begin{proposition}\label{prp:branch-set-homology}The reduced Khovanov homology of the branch set $\tau_t(\ell)$ is described in Figure \ref{fig:Kh-torus-link} for $t\ge0$. \end{proposition}

\begin{figure}[ht!]
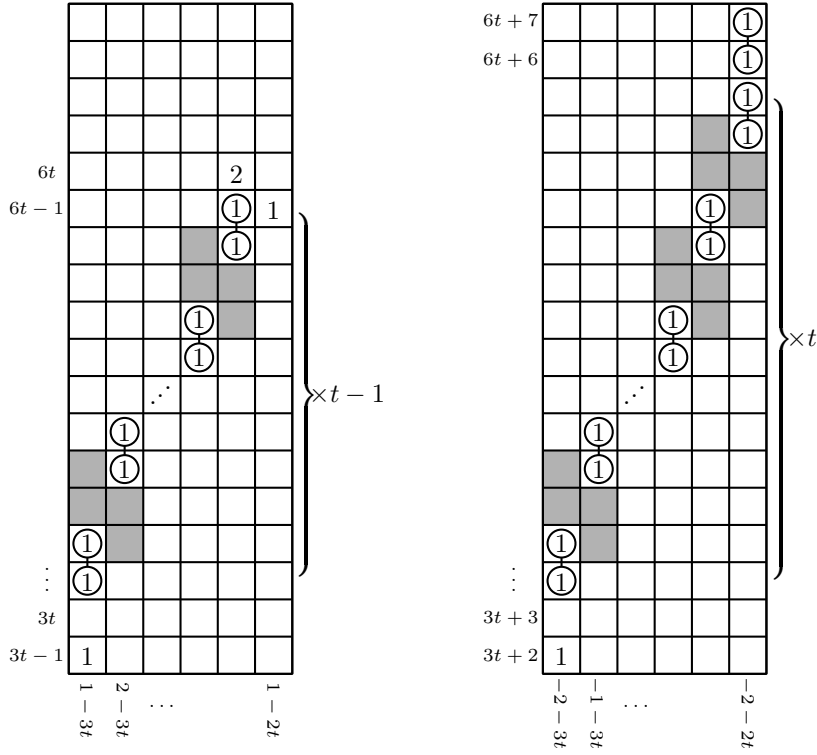

\begin{center}
\labellist\small
	\pinlabel $1$ at 324 413
	\pinlabel $1$ at 324 485
	\pinlabel $1$ at 324 522
	
	\pinlabel $1$ at 360 594
	\pinlabel $1$ at 360 630
		
	\pinlabel $1$ at 432 702	
	\pinlabel $1$ at 432 739	
		
	\pinlabel $1$ at 468 810	
	\pinlabel $1$ at 468 846

	\pinlabel $2$ at 468 879
	
	\pinlabel $1$ at 504 844

	\pinlabel \rotatebox{90}{$\underbrace{\phantom{aaaaaaaaaaaaaaaaaaaaaaaaaa}}$} at 530 666
	\pinlabel $\times t-1$ at 575 666
	\pinlabel \rotatebox{-45}{$\vdots$} at 400 671

	\tiny
	\pinlabel \rotatebox{-90}{$1-3t$} at 322 361
	\pinlabel \rotatebox{-90}{$2-3t$} at 357 361
	\pinlabel $\hdots$ at 396 365
	\pinlabel \rotatebox{-90}{$1-2t$} at 503 361
	
	\pinlabel $3t-1$ at 275 413
	\pinlabel $3t$ at 285 450
	\pinlabel $\vdots$ at 285 495
	
	\pinlabel $6t-1$ at 275 845
	\pinlabel $6t$ at 285 881
	
	\endlabellist
\raisebox{0pt}{\includegraphics[scale=0.4]{grids/full-twist-grid}}
\qquad\qquad\qquad\qquad
\labellist\small
	\pinlabel $1$ at 324 413
	\pinlabel $1$ at 324 485
	\pinlabel $1$ at 324 522
	
	\pinlabel $1$ at 360 594
	\pinlabel $1$ at 360 630
		
	\pinlabel $1$ at 432 702	
	\pinlabel $1$ at 432 739	
		
	\pinlabel $1$ at 468 810	
	\pinlabel $1$ at 468 846
		
	\pinlabel $1$ at 504 917
	\pinlabel $1$ at 504 954
	\pinlabel $1$ at 504 990
	\pinlabel $1$ at 504 1027

	\pinlabel \rotatebox{90}{$\underbrace{\phantom{aaaaaaaaaaaiaaaaaaaaaaaaaaaaaataaaa}}$} at 528 720
	\pinlabel $\times t$ at 556 720

	\pinlabel \rotatebox{-45}{$\vdots$} at 400 671

	\tiny
	\pinlabel \rotatebox{-90}{$-2-3t$} at 322 358
	\pinlabel \rotatebox{-90}{$-1-3t$} at 357 358
	\pinlabel $\hdots$ at 396 365
	\pinlabel \rotatebox{-90}{$-2-2t$} at 503 358
	
	\pinlabel $3t+2$ at 275 413
	\pinlabel $3t+3$ at 275 450
	\pinlabel $\vdots$ at 278 495
	
	\pinlabel $6t+6$ at 275 990
	\pinlabel $6t+7$ at 275 1027

	\endlabellist
\raisebox{0pt}{\includegraphics[scale=0.4]{grids/positive-surgery-grid}}
\end{center}
\vspace*{20pt}
\caption{The reduced Khovanov homology of the $(3,3t)$-torus link as computed in Proposition \ref{prp:torus} (left), and the reduced Khovanov homology of the branch set $\tau_t(\ell)$ (right).}
\label{fig:Kh-torus-link}
\end{figure}

\begin{proof} The strategy of proof is to iterate Proposition \ref{prp:iteration}, as in the schematic of Figure \ref{fig:main-scheme}, to the closure of the braid $\si_1\si_2\si_1^3\si_2\Delta^t$:
\[
\xymatrix@C=50pt@R=10pt{
{\underline{\si_1}\si_2\si_1^3\si_2\Delta^t}\ar[r]^{\text{oriented}}_{\text{resolution}} &  {\si_2\underline{\si_1^3}\si_2\Delta^t}\ar[r]^{\text{oriented}}_{\text{resolution}} & {\underline{\si_2^2}\Delta^t}\ar[r]^{\text{oriented}}_{\text{resolution}} & {\Delta^t}
}\]
It is straightforward to verify that the constants and unoriented resolutions are as claimed in Figure \ref{fig:main-scheme}. The relevant piece of the group corresponding to each of the following steps is summarized in Figure \ref{fig:proof-steps}.

\begin{figure}[ht!]
\begin{center}\quad
\xymatrix@C=-50pt@R=20pt{
	&
	{\labellist\small
	\pinlabel $\Delta^t$ at 68 2
	\endlabellist\raisebox{0pt}{\includegraphics[scale=0.4]{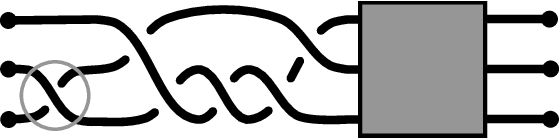}}}\ar@/^1pc/[dr] 
	&&&\\
	{\labellist\small
	\pinlabel $U$ at -150 -30
	\pinlabel $\Delta^t$ at 68 2
	\endlabellist\raisebox{0pt}{\includegraphics[scale=0.4]{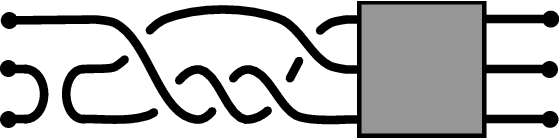}}}\ar@/^1pc/[ur]^-{c=4t+2}_{\times 1}
	&&
	{\labellist\small
	\pinlabel $\Delta^t$ at 69 3
	\endlabellist\raisebox{0pt}{\includegraphics[scale=0.4]{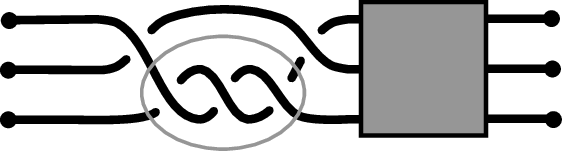}}}\ar@/^1pc/[dr]
	 &&\\
	&
	{\labellist\small
	\pinlabel $\Delta^t$ at 69 2
	\pinlabel $H^+$ at -160 -30
	\endlabellist\raisebox{0pt}{\includegraphics[scale=0.4]{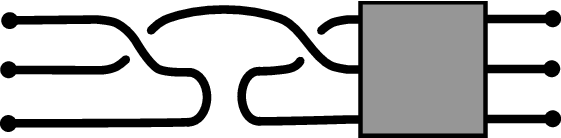}}}\ar@/^1pc/[ur]^-{c=4t}_{\times3}
	&&
	{\labellist\small
	\pinlabel $\Delta^t$ at 69 1
	\endlabellist\raisebox{0pt}{\includegraphics[scale=0.4]{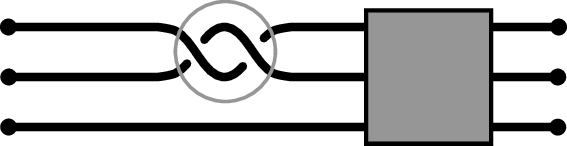}}}\ar@/^1pc/[dr]
	 &\\
	&&
	{\labellist\small
	\pinlabel $U\sqcup U$ at -173 -30
	\pinlabel $\Delta^t$ at 69 2
	\endlabellist\raisebox{0pt}{\includegraphics[scale=0.4]{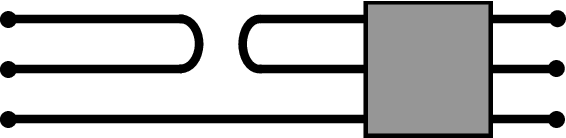}}}\ar@/^1pc/[ur]^-{c=4t}_{\times2}
	&&
	{\labellist\small
	\pinlabel $\Delta^t$ at 69 2
	\endlabellist\raisebox{0pt}{\includegraphics[scale=0.4]{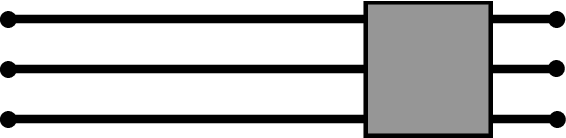}}}
	}
\end{center}
\caption{Schematic of the proof of Proposition \ref{prp:branch-set-homology}: at each step the number $\times n$ indicates the number of iterations of the long exact sequence using Proposition \ref{prp:iteration} to produce $E_1(\si_1\si_2\si_1^3\si_2\Delta^t)$ computing $\Khred(\tau_t(\ell))$.  Recall that $H^+$ denotes the positively clasped Hopf link, with $\Khred(H^+)$ as in Figure \ref{fig:small-homology}.}
\label{fig:main-scheme}
\end{figure}

\begin{figure}[ht!]
\begin{center}
\labellist\small
	\pinlabel $1$ at 324 413
	
	\pinlabel $1$ at 359 487
	\pinlabel $1$ at 359 522
	\pinlabel $2$ at 359 557
	
	\pinlabel $1$ at 395 522
	
	\pinlabel $1$ at 359 595
	\pinlabel $1$ at 359 631
	
	\pinlabel $1$ at 395 557
	\pinlabel $1$ at 395 595

	\tiny
	
	\endlabellist
\raisebox{0pt}{\includegraphics[scale=0.4]{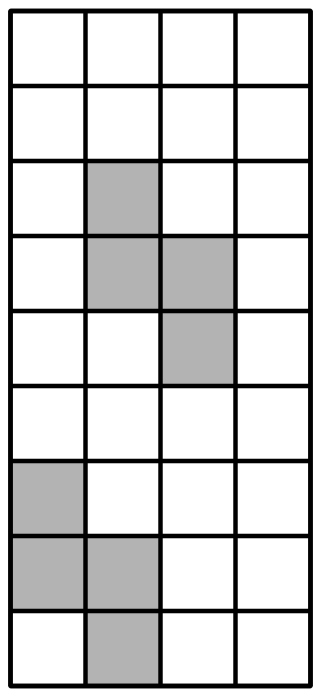}}\qquad\qquad\qquad
\labellist\small
	\pinlabel $1$ at 324 413
	
	\pinlabel $1$ at 359 487
	\pinlabel $1$ at 359 522
	\pinlabel $2$ at 359 557
	
	\pinlabel $1$ at 394 522
	
	\pinlabel $1$ at 359 595
	\pinlabel $1$ at 359 631
	
	\pinlabel $2$ at 395 557
	\pinlabel $2$ at 395 595
	\pinlabel $2$ at 395 629
	\pinlabel $1$ at 395 666
	\pinlabel $1$ at 395 702
	
	\tiny
		
	\endlabellist
\raisebox{0pt}{\includegraphics[scale=0.4]{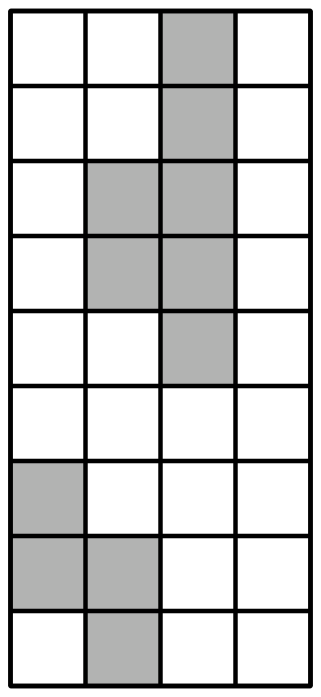}}\qquad\qquad\qquad
\labellist\small
	\pinlabel $1$ at 324 413
	
	\pinlabel $1$ at 359 487
	\pinlabel $1$ at 359 522
	\pinlabel $2$ at 359 557
	
	\pinlabel $1$ at 394 522
	
	\pinlabel $1$ at 359 595
	\pinlabel $1$ at 359 631
	
	\pinlabel $2$ at 395 557
	\pinlabel $2$ at 395 595
	\pinlabel $2$ at 395 629
	\pinlabel $1$ at 395 666
	\pinlabel $1$ at 395 702
	
	\pinlabel $1$ at 431 595

	\tiny
	\pinlabel $\vdots$ at 280 630
	\pinlabel $6t+6$ at 275 666
	\pinlabel $6t+7$ at 275 702
	
	\pinlabel $\cdots$ at 355 370
	\pinlabel \rotatebox{-90}{$-2-2t$} at 395 357
	\pinlabel \rotatebox{-90}{$-1-2t$} at 431 357

\endlabellist
\raisebox{0pt}{\includegraphics[scale=0.4]{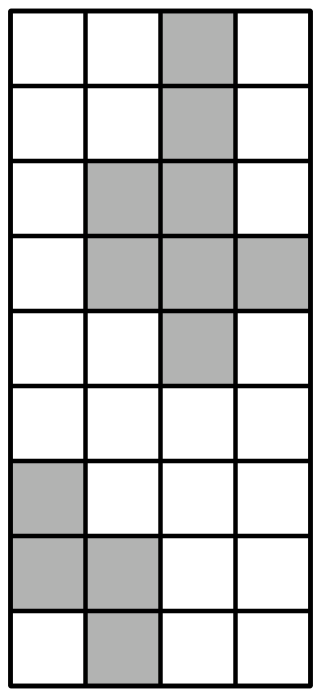}}
\end{center}
\vspace*{20pt}
\caption{The potential new generators (shaded) appearing in $E_1(\si_2^2\Delta^t)$, $E_1(\si_2\si_1^3\si_2\Delta^t)$ and $E_1(\si_1\si_2\si_1^3\si_2\Delta^t)$ (left to right) when using  Proposition \ref{prp:iteration} to compute $\Khred(\overline{\si_2^2\Delta^t})$, $\Khred(\overline{\si_2\si_1^3\si_2\Delta^t})$ and $\Khred(\overline{\si_1\si_2\si_1^3\si_2\Delta^t})$.}
\label{fig:proof-steps}
\end{figure}

\begin{align*}
E_1(\si_2^2\Delta^t) 
&\cong 
\Big(\Khred(\overline{\Delta^t})\oplus\big(\bigoplus_{q=0}^{1}\Khred(U\sqcup U)[\textstyle-2t+\half,6t+\half+q]\big)\Big)[\textstyle-1,1]
\\
E_1(\si_2\si_1^3\si_2\Delta^t) 
&\cong
\Big(\Khred(\overline{\si_2^2\Delta^t})\oplus\big(\bigoplus_{q=0}^{2}\Khred(H^+)[\textstyle-2t+\half,6t+\half+q]\big)\Big)[\textstyle-\frac{3}{2},\frac{3}{2}]
\\
E_1(\si_1\si_2\si_1^3\si_2\Delta^t) 
&\cong
\Big(\Khred(\overline{\si_2\si_1^3\si_2\Delta^t})\oplus\Khred(U)[\textstyle -2t-1+\half,6t+4-\half]\Big)[\textstyle-\frac{1}{2},\frac{1}{2}]
\end{align*}
This illustrates how Proposition \ref{prp:iteration} is used to generate a collection of new generators, and these have been highlighted in Figure \ref{fig:proof-steps}. In particular, $\Khred(\tau_t(\ell))$ may be computed by considering 
\[E_1(\si_1\si_2\si_1^3\si_2\Delta^t)
\cong\Khred(T_{3,3t})[3,3]
\oplus\left(\begin{matrix}
\bigoplus_{q=0}^{1}\Khred(U\sqcup U)[\textstyle-2t+\half,6t+\half+q][\textstyle-3,3] \\
\bigoplus_{q=0}^{2}\Khred(H^+)[\textstyle-2t+\half,6t+\half+q][\textstyle-2,2] \\
\Khred(U)[-2t-1,6t+4]\end{matrix}\right)\] 
in combination with $\Khred'(\tau_t(\ell))$. Notice that only 11 possible new generators appear (see the right-most shaded group described in Figure \ref{fig:proof-steps}).   

Since $\tau_t(\ell)$ is a knot, the single generator of $\Khred'(\tau_t(\ell))$ appears in grading $(-2-3t,2+3t)$ (see Figure \ref{fig:Kh-torus-link}). As a result, the portion of the homology group shown in Figure \ref{fig:proof-steps} must collapse in $\Khred'(\tau_t(\ell))$, placing constraints on the potential differentials calculating $\Khred(\tau_t(\ell))$. This analysis is shown in Figure \ref{fig:pair-steps} and described below.

\begin{figure}[ht!]
\begin{center}
\labellist\small
	\pinlabel $(a)$ at 360 360

	\pinlabel $1$ at 324 414
	\pinlabel $1$ at 324 450
	\pinlabel $2$ at 324 487
	\pinlabel $1$ at 324 522
	\pinlabel $1$ at 324 557
	
	\pinlabel $1$ at 360 450
	\pinlabel $2$ at 360 487
	\pinlabel $2$ at 360 522
	\pinlabel $2$ at 360 557
	
	\pinlabel $1$ at 395 522
	
	\pinlabel $1$ at 360 595
	\pinlabel $1$ at 360 631
	
	\pinlabel $\Longrightarrow$ at 460 520	
	\tiny
	
	\endlabellist
\raisebox{0pt}{\includegraphics[scale=0.4]{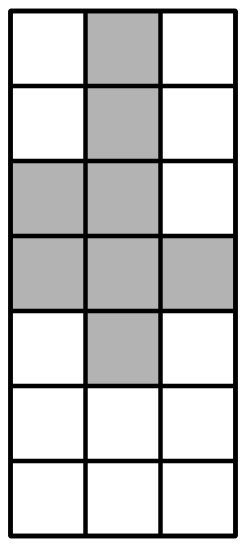}}\qquad\quad
\labellist\small
	\pinlabel $(b)$ at 360 360
	\pinlabel $1$ at 324 414
	\pinlabel $1$ at 324 450
	\pinlabel $2$ at 324 487
	\pinlabel $1$ at 324 522
	\pinlabel $1$ at 324 557
	
	\pinlabel $1$ at 360 450
	\pinlabel $2$ at 360 487
	\pinlabel $2$ at 360 522
	\pinlabel $2$ at 360 557
	
	\pinlabel $1$ at 395 522
	
	\pinlabel $1$ at 360 595
	\pinlabel $1$ at 360 631
	
	\pinlabel $\to$ at 379 520
	\pinlabel $\Longrightarrow$ at 460 520	

	\endlabellist
\raisebox{0pt}{\includegraphics[scale=0.4]{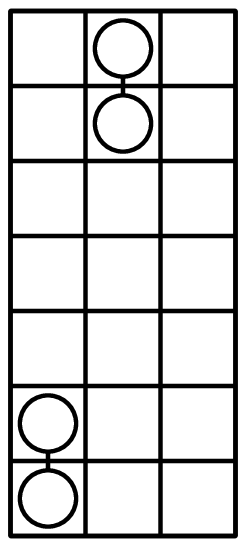}}\qquad\quad
\labellist\small
	\pinlabel $(c)$ at 360 360
	\pinlabel $1$ at 324 414
	\pinlabel $1$ at 324 450
	\pinlabel $2$ at 324 487
	\pinlabel $1$ at 324 522
	\pinlabel $1$ at 324 557
	
	\pinlabel $1$ at 360 450
	\pinlabel $2$ at 360 487
	\pinlabel $1$ at 360 522
	\pinlabel $2$ at 360 557

	\pinlabel $1$ at 360 595
	\pinlabel $1$ at 360 631

	\pinlabel $\to$ at 341 555
	\pinlabel $\Longrightarrow$ at 460 520	
	
	\endlabellist
\raisebox{0pt}{\includegraphics[scale=0.4]{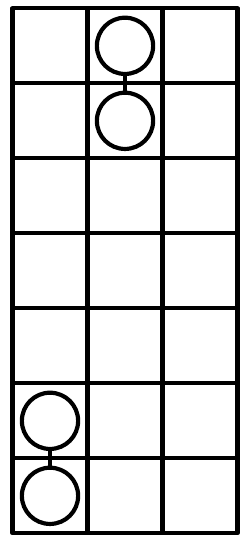}}\qquad\quad
\labellist\small
	\pinlabel $(d)$ at 360 360
	\pinlabel $1$ at 324 414
	\pinlabel $1$ at 324 450
	\pinlabel $2$ at 324 487
	\pinlabel $1$ at 324 522
	
	\pinlabel $1$ at 360 450
	\pinlabel $2$ at 360 487
	\pinlabel $1$ at 360 522
	\pinlabel $1$ at 360 557

	\pinlabel $1$ at 360 595
	\pinlabel $1$ at 360 631
	
	\pinlabel $\to$ at 341 485	
	\pinlabel $\Longrightarrow$ at 460 520	
	
\endlabellist
\raisebox{0pt}{\includegraphics[scale=0.4]{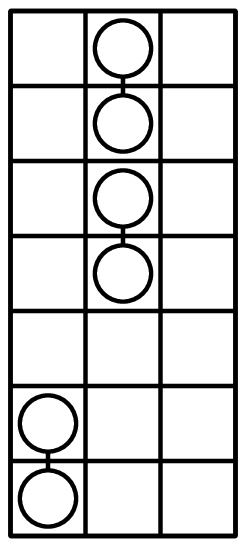}}\qquad\quad
\labellist\small
	\pinlabel $(e)$ at 360 360
	\pinlabel $1$ at 324 414
	\pinlabel $1$ at 324 450

	\pinlabel $1$ at 360 522
	\pinlabel $1$ at 360 557

	\pinlabel $1$ at 360 595
	\pinlabel $1$ at 360 631

\endlabellist
\raisebox{0pt}{\includegraphics[scale=0.4]{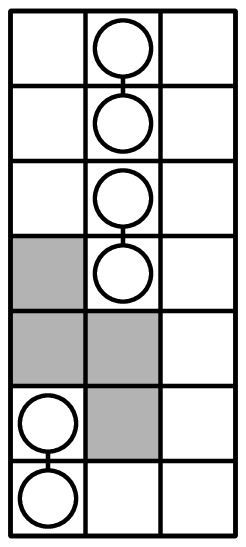}}
\end{center}
\vspace*{10pt}
\caption{Using the pairing computing $\Khred'(\tau_t(\ell))$ to determine $\Khred(\tau_t(\ell))$ up to indeterminate summands.}
\label{fig:pair-steps}
\end{figure}

First notice that the top- and bottom-most generators must each pair as in Figure \ref{fig:pair-steps}(b). Since this leaves nothing with which the right-most generator can pair, there must be a differential canceling this generator. Similarly, as shown in Figure \ref{fig:pair-steps}(c), the upper most $\bF^2$ cannot be cancelled by any higher differential as is required; this forces a second non-trivial differential as shown and determines the pairing in Figure \ref{fig:pair-steps}(d). 

This determines the homology, up to indeterminate summands. In particular, as in the proof of Claim \ref{cl:2}, there is a single differential of the form $\bF^2\to\bF^2$ that is non-trivial whereby either a single pair or all 4 generators are cancelled (see Figure \ref{fig:pair-steps}(d)). This gives rise to the $t^{\rm th}$ indeterminate summand shown in Figure \ref{fig:Kh-torus-link}.
\end{proof}

\subsection{A lower bound for homological width} Much of the strength of Khovanov homology is retained by relaxing the absolute $\half\bZ\oplus\half\bZ$-grading to a relative $\bZ\oplus\bZ$-grading. Moreover, by ignoring the secondary $q$-grading, the reduced Khovanov homology takes the form $\Khred(L)\cong\bigoplus_{\delta=1}^{w(L)}(\bF^{b_\delta})_{\delta}$; the positive integer $w(L)$ is the homological width of the link $L$ (as in Definition \ref{def:width}). Given this relatively $\bZ$-graded version, it is easy to verify that $\det(L)=\big|\sum_\delta (-1)^\delta b_\delta\big|=\chi(\Khred(L))$ (see \cite[Proposition 2.2]{Watson2008}, for example). In the present setting, we have the following consequence of Proposition \ref{prp:branch-set-homology}:
\[\Khred(\tau_t(\ell))\cong (\bF^{b_1})_1\oplus\cdots\oplus(\bF^{b_{t+1}})_{t+1}\]
as a relatively $\bZ$-graded group, where $b_i>0$ for each $1\le i\le t+1$ (of course, due to indeterminate summands, the integers $b_i$ have not been determined exactly). 

\begin{proposition}\label{prp:ell-plus-one}
The reduced Khovanov homology of the branch set $\tau_t(\ell+1)$, as a relatively $\bZ$-graded group, is 
\[\Khred(\tau_t(\ell))\cong (\bF^{b_1})_1\oplus\cdots\oplus(\bF^{b_{t+1}})_{t+1}\oplus(\bF)_{t+2}.\] Moreover, the quantum grading of the generator in the $(t+2)^{\text nd}$ (relative) grading, is strictly smaller than the largest quantum grading in the $(t+1)^{\text st}$ (relative) grading supporting non-trivial homology 
\end{proposition}
\begin{proof} First we determine $\Khred(\tau(\ell+1))$ as an absolutely bigraded group. Applying the mapping cone (\ref{eqn:cone}), 
\begin{align*} \Khred(\tau(\ell+1))&\cong H_*\left( \Khred(\tau(\ell))[\textstyle-\half,\half]\to\Khred(U)[\textstyle -2t-\frac{3}{2},6t+\frac{11}{2}]\right)\\
&\cong H_*\left( \Khred(\tau(\ell)) \to (\bF)^{\delta=-2t-1}_{q=6t+5}\right)[\textstyle -\half,\half]\end{align*} since $c=4t+3$. Notice that $\Khred^{-2-2t}_{6t+7}(\tau(\ell))\cong \bF$ from Figure \ref{fig:Kh-torus-link}, verifying the second claim of the proposition.

To verify the first claim, we analyze the potential new generator as in Figure \ref{fig:surgery-sample} (note that the pairings in the portion the group that is not displayed are uniquely determined). In particular, by applying Theorem \ref{thm:turner} we have that $\Khred'(\tau(\ell+1))\cong\bF^2$ with generators supported in $\delta+q=0$ and $\delta+q=4t+4$. The latter must be the generator that does not pair with the top-most generator in Figure \ref{fig:surgery-sample} (for absolute gradings, consult Figure \ref{fig:Kh-torus-link}). Despite this indeterminacy  Turner's spectral sequence, as in the proof of Proposition \ref{prp:branch-set-homology}, ensures that this new generator is present and \[\Khred(\tau_t(\ell+1))\cong(\bF^{b_1})_1\oplus\cdots\oplus(\bF^{b_{t+1}})_{t+1}\oplus(\bF)_{t+2}\] as a relatively $\bZ$-graded group. 
\end{proof}

\begin{figure}[ht!]\begin{center}
\labellist\small
	\pinlabel $1$ at 324 414
	\pinlabel $1$ at 324 450

	\pinlabel $1$ at 360 522
	\pinlabel $1$ at 360 557

	\pinlabel $1$ at 360 595
	\pinlabel $1$ at 360 631
	
	\pinlabel $1$ at 395 557
		\pinlabel $\Longrightarrow$ at 475 520	

\endlabellist
\raisebox{0pt}{\includegraphics[scale=0.4]{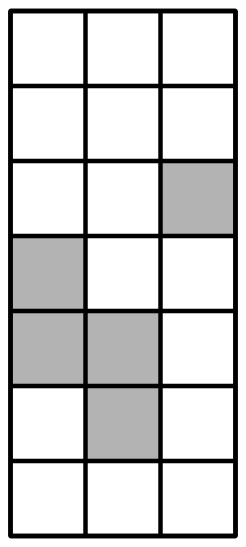}}
\qquad
\qquad
\labellist\small
	\pinlabel $1$ at 324 414
	\pinlabel $1$ at 324 450

	\pinlabel $1$ at 360 522
	\pinlabel $1$ at 360 557
	
	\pinlabel $1$ at 360 595
	\pinlabel $1$ at 360 631
		\pinlabel $1$ at 395 557

\endlabellist
\raisebox{0pt}{\includegraphics[scale=0.4]{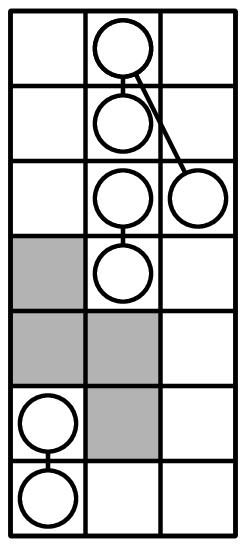}}
\end{center}
\caption{Determining $\Khred(\tau_t(\ell+1))$, showing only the (relative) $\delta$-gradings $t,t+1,t+2$. Notice that the $(t+2)^{\rm nd}$ $\delta$-grading is forced to appear in the group  $\Khred(\tau_t(\ell+1))$, despite the indeterminacy in the pairing for higher differentials as shown.
}
\label{fig:surgery-sample}
\end{figure}

\begin{corollary}\label{crl:generic}The tangle $(B^3,\tau_t)$ is generic in the sense of \cite[Definition 5.15]{Watson2008}.\end{corollary}
\begin{proof}Each of the $b_i$ in the expression for $\Khred(\tau_t(\ell+1))$ are positive so there are no blank diagonals (see \cite[Remark 4.19]{Watson2008}). Note that $w(\tau(\ell+1))=w(\tau(\ell))+1$; the second part of Proposition \ref{prp:ell-plus-one} verifies expansion long form for $\tau(\ell+1)$, while $\tau(\ell)$ has expansion short form by Proposition \ref{prp:branch-set-homology} \cite[Definition 5.3]{Watson2008}. These observations ensure that the tangle is expansion generic \cite[Definition 5.7]{Watson2008}, hence generic \cite[Definition 5.15]{Watson2008}. \end{proof}

\begin{proof}[Proof of Proposition \ref{prp:branch-width}] First notice that $w(\tau_t(\ell))=t+1\le w(\tau_t(n))\le t+2=w(\tau_t(\ell+1))$ for all $n$, by \cite[Lemma 4.20]{Watson2008}; we have identified the single integer framing at which the branch sets  change homological width. By Corollary \ref{crl:generic} the tangle is generic. Hence the minimum width on branch sets corresponding to integer surgeries provides a lower bound for the width of a branch set associated with any non-trivial surgery \cite[Theorem 5.16]{Watson2008}. That is, $w_{K_t} = t+1$.
\end{proof}

This completes the proof of Theorem \ref{thm:twist}. Notice that in the case $t=1$ (where $K_1$ is the figure eight knot), we only have $w(\tau_1(n))>2$ for $n\ge0$. From this it follows that $w(\tau_1(r))>2$ for any rational $r\ge0$, confirming that $\pi_1(S^3_r(K_1))$ must be infinite for non-negative surgery coefficients  (see \cite[Theorem 7.2]{Watson2008}). However $S^3_r(K_1)\cong S^3_{-r}(K_1)$ since $K_1$ is amphichiral.

\begin{remark}\label{rmk:det} In general \[\Khred(\tau_t(\ell+m))\cong (\bF^{b_1})_1\oplus\cdots\oplus (\bF^{b_t})_t\oplus(\bF^{b^*_{t+1}})_{t+1}\oplus(\bF^{b_{t+2}})_{t+2} \] where $0<b_{t+1}^*\le b_{t+1}$ and $0<b_{t+2}\le m$, as can be seen from the spectral sequence with $E_1(\si_1^{1+m}\si_2\si_1^3\si_2\Delta^t)$ (resolving the first $m$ crossings) following an argument similar to the proof of Proposition \ref{prp:ell-plus-one}. Note that setting \[m=\begin{cases}1 & {\rm for\ } t {\rm \ odd} \\ 5 & {\rm for\ } t {\rm \ even}\end{cases}\] gives \[\det(\tau_t(0))=\chi(\Khred(\tau(0))=\chi\left( E_1(\si_1^{1+m}\si_2\si_1^3\si_2\Delta^t)\right)= 0\] as claimed (and left to the reader) in the proof of Proposition \ref{prp:framed-tangle}.\end{remark}

This calculation of $\Khred(\tau_t(\ell+m))$ should be compared with \cite[Lemma 4.10]{Watson2008} which establishes the same stable behaviour for tangles associated with a strongly invertible knot in general. While we have appealed to the machinery of \cite{Watson2008} in order to establish the width bound of Proposition \ref{prp:branch-width}, the reader should be assured that the computations required for this paper could be made entirely self contained. One need only fix a convention for the branch sets $\tau_t(r)$ when $r\in\bQ$ (as in \cite[Section 3]{Watson2008}), and iterate the mapping cone from Section \ref{sub:mc}. The condition that the tangle be {\em generic} simply ensures that this calculation is possible without the need to calculate any differentials. 

It seems worth noting that results of Boyer and Zhang ensure that a finite filling on a knot in $S^3$ can only occur with integer or half-integer surgery coefficient  \cite{BZ1996}. While our proof does not appeal to this fact, it is interesting that the Khovanov homology of only two branch sets associated with integer surgeries suffice to prove Theorem \ref{thm:twist}.

\frenchspacing

\bibliographystyle{QT}
\bibliography{twist-arxiv}

\end{document}